\documentclass[12pt,reqno]{amsart}

\usepackage{amssymb,hyperref}
\usepackage{amsmath}
\usepackage{amsfonts}
\usepackage[all]{xy}

\newtheorem{theorem}{Theorem}[section]
\newtheorem{proposition}[theorem]{Proposition}
\newtheorem{lemma}[theorem]{Lemma}
\newtheorem{corollary}[theorem]{Corollary}

\theoremstyle{definition}
\newtheorem{definition}[theorem]{Definition}
\newtheorem{remark}[theorem]{Remark}
\newtheorem{example}[theorem]{Example}
\numberwithin{equation}{section}

\newcommand{\C}{\mathbb{C}}
\newcommand{\R}{\mathbb{R}}

\newcommand{\Q}{\mathbb{Q}}
\newcommand{\Z}{\mathbb{Z}}
\newcommand{\g}{\mathfrak{g}}

\begin{document} 

\title[Higgs bundles and flat connections on Sasakian manifolds, II]{Higgs bundles and 
flat connections over compact Sasakian manifolds, II: quasi-regular bundles}

\author[I. Biswas]{Indranil Biswas}

\address{Department of Mathematics, Shiv Nadar University, NH91, Tehsil
Dadri, Greater Noida, Uttar Pradesh 201314, India}

\email{indranil29@gmail.com, indranil@math.tifr.res.in}

\author[H. Kasuya]{Hisashi Kasuya}

\address{Department of Mathematics, Graduate School of Science, Osaka University, Osaka,
Japan}

\email{kasuya@math.sci.osaka-u.ac.jp}



\begin{abstract}
In this continuation of \cite{BK} and \cite{BM} we investigate the non-abelian Hodge 
correspondence on compact Sasakian manifolds with emphasis on the quasi-regular case.
We introduce on quasi-regular Sasakian manifolds the notions of quasi-regularity and 
regularity of basic vector bundles. These notions are useful in relating the vector bundles
over a quasi-regular Sasakian manifold with the orbibundles over the orbifold defined
by the orbits of the Reeb foliation of the quasi-regular Sasakian manifold. We note that
the non-abelian Hodge correspondence on any quasi-regular Sasakian manifolds gives a canonical 
correspondence between the semi-simple representations of the orbifold fundamental group 
and the Higgs orbibundles on locally cyclic complex orbifold admitting Hodge metrics. 
Under the assumption of
quasi-regularity of Sasakian manifolds and vector bundles, we extend this 
correspondence to one between the flat bundles and the basic Higgs bundles. We also prove 
a Sasakian analogue of the characterization of numerically flat bundles given by Demailly, 
Peternell and Schneider.
\end{abstract}

\maketitle

\tableofcontents

\section{Introduction}

In \cite{BK}, the authors proved a non-abelian Hodge correspondence on compact Sasakian 
manifolds as an odd-dimensional analogue of the ground-breaking works of Corlette 
\cite{Cor} and Simpson on bundles on compact K\"ahler manifolds
\cite{Si1, Si2}. More precisely, it was shown in \cite{BK} that the harmonic metrics 
on the semi-simple flat bundles over compact Sasakian manifolds provide basic Higgs bundles 
(called b-Higgs bundles in the main text of this paper) over the same compact Sasakian manifold, 
and furthermore, stable basic Higgs bundles over any compact Sasakian manifold admit basic Hermitian 
metrics satisfying an Hermitian--Yang--Mills type equation. Consequently, for a compact 
Sasakian manifold, there is an equivalence between the category of semi-simple flat vector 
bundles on it and the category of polystable Higgs bundles on it with trivial first and 
second basic Chern classes. The purpose of this paper is to study this correspondence in 
the context of the geometric features of Sasakian manifolds.

Sasakian geometry is often qualified as an odd-dimensional ``analogy" of K\"ahler 
geometry. But, recent researches testify that Sasakian geometry is not just an analogue. For example, 
Sasakian geometry gives new aspects of Einstein-metrics (see e.g. \cite{BoG1, BGN, BGK} and 
\cite[Chapter 11]{BoG}). It would not be unreasonable
to hope that our correspondence in \cite{BK} contributes in 
giving new aspects of representations of fundamental groups and Higgs bundles in the
Sasakian context.

Through a series of examples, we provide a comparison between the K\"ahler case and the
Sasakian case. Consider the simplest non-trivial example observed by Hitchin,
\cite{Hit}, on a compact Riemann surface $X$ of genus $g\,\ge\, 2$, where the
Higgs bundle $(E,\,\theta)$ consists of
\begin{itemize}
\item the holomorphic vector bundle $E\,=\,\Omega_{X}^{\frac{1}{2}}\oplus
\Omega_{X}^{-\frac{1}{2}}$ with $\Omega_{X}^{\frac{1}{2}}$
being a square-root of the canonical bundle $\Omega_{X}$ of $X$ (namely, a
theta characteristic), and 
\item $\theta\,=\,\left(
\begin{array}{cc}
0&0 \\
1 & 0
\end{array}
\right)$ as in the proof of \cite[Corollary 4.23]{Hit}.
\end{itemize}
This $(E,\,\theta)$ corresponds to a representation $\Gamma\,:=\,\pi_1(X)\, \longrightarrow\, {\rm 
SL}_{2}(\R)$ which is one of the $2^{2g}$ liftings of the natural representation of 
$\Gamma$ into ${\rm PSL}_{2}(\R)$ given by the
uniformization theorem applied to $X$. For the Sasakian case, we consider the universal 
covering $\widetilde{\rm SL}_{2}(\R)$ of ${\rm PSL}_{2}(\R)$. Let $ 
\widetilde{\Gamma}\,\subset\, \widetilde{\rm SL}_{2}(\R)$ be a discrete subgroup such that the 
corresponding quotient $\widetilde{\Gamma}\backslash\widetilde{\rm SL}_{2}(\R)$ is compact. 
Since ${\rm SL}_{2}(\R)$ is a double covering of ${\rm PSL}_{2}(\R)$, we have the canonical 
representation $ \widetilde{\Gamma}\,\longrightarrow\, {\rm SL}_{2}(\R)$. As an example of 
the correspondence in \cite{BK}, we give a correspondence between this representation of
$\widetilde{\Gamma}$ and an explicit basic Higgs bundle over 
the compact Sasakian manifold $\widetilde{\Gamma}\backslash \widetilde{\rm SL}_{2}(\R)$ whose
fundamental group is $\widetilde{\Gamma}$ (see Examples \ref{exsl1}, \ref{Exsl22}, 
\ref{PSL2g} and \ref{PSL2g2}).

For the fundamental group $\Gamma\,\subset\,{\rm PSL}_{2}(\R)$
of a compact hyperbolic (orbifold) Riemann surface $X$, we can take $\widetilde{\Gamma}$ to be
the inverse image of $\Gamma$ for the natural projection of $\widetilde{\rm SL}_{2}(\R)$
to ${\rm PSL}_{2}(\R)$. In this particular case we have $\widetilde{\Gamma}
\backslash \widetilde{\rm SL}_{2}(\R)\,=\,\Gamma \backslash{\rm PSL}_{2}(\R)$.

Undoubtedly, the notion of quasi-regularity is the most important condition on compact Sasakian 
manifolds. We recall that a compact Sasakian manifold is quasi-regular if every orbit of 
its canonical foliation (Reeb foliation) is closed. In this paper, we define the 
quasi-regularity, and also regularity, of basic vector bundles over quasi-regular 
compact Sasakian manifolds; this is carried out in Section \ref{q-rv}. These notions contribute 
in refining both the set-up and the arguments in \cite{BM}.

We notice that the orbit 
space of the Reeb foliation on a quasi-regular compact Sasakian manifold can be seen as a 
locally cyclic complex orbifold admitting a K\"ahler metric. For example, for the above 
example of $\Gamma \backslash{\rm PSL}_{2}(\R)$, the orbifold Riemann surface 
$\Gamma\backslash H$ is the orbit
space of the Reeb foliation. Any Sasakian structure has the realization as a smoothing of a 
locally cyclic Hodge orbifold \cite{BoG}. One of our motivations for defining the 
quasi-regularity and the regularity of basic vector bundles is to study the K\"ahler 
orbifolds in terms of Sasakian geometry. Orbibundles over the orbit space of a 
quasi-regular compact Sasakian manifold can be seen as the regular basic bundles over the 
quasi-regular compact Sasakian manifold. In view of this, as a consequence of our results 
in \cite{BK}, we can show the existence of Hermite-Einstein metrics on Higgs orbibundles 
over orbifolds. This enables us to obtain a correspondence between the representations
of the orbifold fundamental group and the Higgs bundles over the orbifold; it may
be mentioned that this is achieved 
directly without treating orbifold singularities (see Theorem \ref{orbiHE} and Corollary 
\ref{orbirephi}).

Unlike the usual Hermitian metrics on holomorphic vector bundles, the existence of basic 
Hermitian metrics on basic vector bundles does not hold in general. This aspect causes several 
difficulties. By a standard argument, our assumption of quasi-regularity of basic vector 
bundles over quasi-regular compact Sasakian manifolds has an advantage regarding the question 
of the existence of basic Hermitian metrics, in fact, the quasi-regularity of basic vector 
bundles ensures that a basic Hermitian metric exists (see Lemma \ref{lemSinv}). This advantage 
ensues that on quasi-regular compact Sasakian manifolds, under the assumption of 
quasi-regularity, we obtain a correspondence between the not-necessarily semi-simple flat 
bundles and a certain class of not-necessarily polystable basic Higgs bundles (see Theorem 
\ref{EQQUA} for a precise statement); we note that this is an analogue of \cite[Lemma 
3.5]{Si2}.

Numerically flat bundles were introduced by Demailly, Peternell and Schneider in 
\cite{DPS}. Under the assumption of quasi-regularity, we also obtain an analogue of the 
characterization of the numerically flat bundles given in \cite{DPS} (see Theorem 
\ref{nflatst} and Corollary \ref{flatnflat}).

\section{Preliminaries on Foliations}
\subsection{Basic vector bundles}

Let $(M,\, \mathcal F)$ be a $C^\infty$ foliated manifold such that the co-rank of 
${\mathcal F}$ is $2q$ (so the dimension of the leaves of $\mathcal F$
is $\dim M-2q$). This foliation $(M, \mathcal F)$ is
called {\em transversely holomorphic} if 
there is a foliation atlas $\{U_{\alpha}\}$ together with local submersions $f_{\alpha}\,:\, 
U_{\alpha}\,\longrightarrow\,\C^{q}$, where each $U_{\alpha}\, \subset\, M$
is an open subset with $\bigcup_\alpha U_{\alpha}\,=\, M$ such that
the transition functions 
$$\tau_{\alpha\beta}\,:\,f_{\beta}(U_{\alpha}\cap U_{\beta}) 
\,\longrightarrow\,f_{\alpha}(U_{\alpha}\cap U_{\beta})$$ satisfy the relations 
$f_{\alpha}\,=\,\tau_{\alpha\beta}f_{\beta}$ with each $\tau_{\alpha\beta}$ being a 
biholomorphism. Let $T\mathcal F\, \subset\, TM$ be the tangent bundle of the foliation 
$\mathcal F$, in other words, $T\mathcal F$ is the $C^\infty$ distribution that gives
$\mathcal F$, and let $$N{\mathcal F}\,=\,TM/T{\mathcal F}$$ be the normal bundle of the
foliation. Then we have the canonical decomposition
\begin{equation}\label{nd}
N{\mathcal F}_{\C}\,=\, N{\mathcal F}\otimes_{\mathbb R}{\C}
\,=\,N^{1,0}{\mathcal F}\oplus N^{0,1}{\mathcal F}
\end{equation}
for which $\overline{N^{1,0}{\mathcal F}}\,=\,N^{0,1}{\mathcal F}$.

On a smooth manifold $M$, if we have a sub-bundle ${\mathcal G}\,\subset\, TM_{\C}
\,=\, TM\otimes_{\mathbb R}{\C}$ of complex co-dimension $q$ so that
\begin{itemize}
\item ${\mathcal G}+\overline{{\mathcal G}}\,=\,TM_{\C}$, and

\item ${\mathcal G}$ is involutive,
\end{itemize}
then $M$ admits a transversely holomorphic foliation $\mathcal F$ such that $$T{\mathcal 
F}_{\C}\,=\, {\mathcal G}\cap \overline{\mathcal G}$$ and $N^{1,0}{\mathcal 
F}\,=\,{\mathcal G}/({\mathcal G}\cap \overline{\mathcal G})$ (see \cite{DK}).

A differential form $\omega$ on $M$ is called {\em basic} if the two equations
\begin{equation}\label{bafo}
i_{X}\omega\,=\,0\,=\, {\mathcal L}_{X}\omega
\end{equation}
hold for any locally defined section $X\,\in\, T\mathcal F$. Since
${\mathcal L}_{X}\,=\, di_X+i_Xd$, the condition in \eqref{bafo} is
equivalent to the following condition:
$$
i_{X}\omega\,=\,0\,=\, i_{X}d\omega.
$$ 
We denote by $A^{\ast}_{B}(M)$ the subspace of basic 
forms in the de Rham complex $A^{\ast}(M)$.
Then $A^{\ast}_{B}(M)$ is actually a sub-complex of the de Rham complex $A^{\ast}(M)$.
Denote by $H_{B}^{\ast}(M)$
the cohomology of the basic de Rham complex $A^{\ast}_{B}(M)$. 
Suppose $(M,\,\mathcal F)$ is transverse holomorphic. Corresponding to the decomposition $N{\mathcal F}_{\C}
\,=\,N^{1,0}{\mathcal F}\oplus N^{0,1}{\mathcal F}$ in \eqref{nd}, we have the bigrading
$$A^{r}_{B}(M)_{\C}\,=\,\bigoplus_{p+q=r} A^{p,q}_{B}(M)$$ as well as the decomposition of the exterior 
differential $$d_{\vert A^{r}_{B}(M)_{\C}}=\partial_{B}+\overline\partial_{B},$$ on $A^{r}_{B}(M)_{\C}$, so that $$\partial_{B}:A^{p,q}_{B}(M)
\,\longrightarrow\, A^{p+1,q}_{B}(M)\,\ \text{ and }\, \
\overline\partial_{B}:A^{p,q}_{B}(M)\,\longrightarrow\, A^{p,q+1}_{B}(M)\, .$$

\begin{definition}\label{dbvb}
A {\em basic vector bundle} $E$ over a foliated manifold $(M,\,\mathcal F)$ is a 
$C^\infty$ vector bundle over $M$ which has local trivializations with respect to an open 
covering $M\,=\,\bigcup_{\alpha} U_{\alpha}$
satisfying the condition that each transition 
function $f_{\alpha\beta}\,:\,U_{\alpha}\cap U_{\beta}\,\longrightarrow\, {\rm 
GL}_{r}(\C)$ is basic on $U_{\alpha}\cap U_{\beta}$, in other words, $f_{\alpha\beta}$ is 
constant on the leaves of the foliation $\mathcal F$.

Assume that $(M,\,\mathcal F)$ is transversely holomorphic. We say that a basic vector 
bundle $E$ is {\em holomorphic} if we can take each transition function $f_{\alpha\beta}$ 
to be transversely holomorphic.
\end{definition}

\begin{definition}
For a basic vector bundle $E$, a differential form $\omega\,\in\, A^{\ast}(M,\,E)$ with 
values in $E$ is called basic if $\omega$ is basic on every $U_{\alpha}$
(see Definition \eqref{dbvb}), meaning 
$$\omega\big\vert_{U_{\alpha}}\,\in\, A^{\ast}_{B}(U_{\alpha})\otimes \C^{r}\, ,$$ for 
every $\alpha$, in terms of the trivialization
of $E\big\vert_{U_\alpha}$ as done in Definition \ref{dbvb}.
\end{definition}

Let $$A^{\ast}_{B}(M,\, E)\,\subset\,A^{\ast}(M,\,E)$$ denote the 
subspace of basic forms in the space $A^{\ast}(M,\,E)$ of differential forms with
values in $E$. Corresponding to the
decomposition $N{\mathcal F}_{\C}\,=\,N^{1,0}{\mathcal F}\oplus N^{0,1}{\mathcal F}$
in \eqref{nd}, we have the bigrading $$A^{r}_{B}(M,\,E)\,=\
\bigoplus_{p+q=r} A^{p,q}_{B}(M,\, E)\, .$$
If $E$ is holomorphic, we can extend the operator $\overline\partial_{B}$ to an operator
$$\overline\partial_{E}\,:\,A^{p,q}_{B}(M,\,E)
\,\longrightarrow\, A^{p,q+1}_{B}(M,\,E)$$ satisfying the equation
$\overline\partial_{E}\circ\overline\partial_{E}\,=\,0$.
Conversely, if $$\nabla\,:\,A^{\ast}(M,\,E)\,\longrightarrow\,
A^{\ast+1}(M,\,E)$$
is a connection on $E$ such that for the decomposition $$\nabla\, =\,\nabla^{\prime} 
+\nabla^{\prime\prime},$$ where $\nabla^{\prime}\,:\,A^{p,q}_{B}(M,\,E) \,\longrightarrow\, 
A^{p+1,q}_{B}(M,\,E)$ and $\nabla^{\prime\prime}\,:\,A^{p,q}_{B}(M,\,E)\,\longrightarrow\, 
A^{p,q+1}_{B}(M,\,E)$, we have $$\nabla^{\prime\prime}\circ\nabla^{\prime\prime}\,=\,0\, ,$$ then 
there exists a unique holomorphic bundle structure on $E$ such that 
$\nabla^{\prime\prime}\,=\,\overline\partial_{E}$, like in the case where the base is a complex 
manifold (see \cite[Proposition 3.7]{Ko}).

\subsection{Flat partial connections}

Let $(M,\,\mathcal F)$ be a foliated manifold and $E$ a ${\mathcal C}^{\infty}$ vector bundle over $M$.
Suppose that $E$ is a basic vector bundle over $(M,\,{\mathcal F})$.
Then, we can define the canonical differential operator
\begin{equation}\label{ej}
D\,:\, {\mathcal C}^{\infty}(E)\,\longrightarrow\,
{\mathcal C}^{\infty}(E\otimes T{ \mathcal F}^{\ast})
\end{equation}
such that for any $X\,\in\, T\mathcal F$, any smooth function $f$ on $M$
and any basic section $s$ of $E$, we have
\[
D_{X}(f s)\,=\,X(f) s.
\]
A section $s$ of $E$ is basic if and only if $Ds\,=\,0$.
In particular, a Hermitian metric $h$ on $E$ is basic if and only if $D$ preserves $h$.
Extend $D$ in \eqref{ej} to $$D\,:\, {\mathcal C}^{\infty}(E\otimes \bigwedge\nolimits^{k} T{\mathcal
F}^{\ast})\,\longrightarrow\,{\mathcal C}^{\infty}(E\otimes \bigwedge\nolimits^{k+1}
T{ \mathcal F}^{\ast});$$
since the sub-bundle $T{ \mathcal F}
\,\subset\, TM$ is involutive, we have $D\circ D\,=\,0$. Conversely, suppose
that we have a flat partial ${\mathcal F}$-connection, i.e., a linear
differential operator $$D\,:\, {\mathcal C}^{\infty}(E)
\,\longrightarrow\,
{\mathcal C}^{\infty}(E\otimes T{ \mathcal F}^{\ast})$$ such that
\begin{itemize}
\item for any $X\in T\mathcal F$, and any smooth function $f$ on $M$, the equation
\[ D_{X}(f s)\,=\,fD_{X} s+ X(f) s
 \]
holds for all smooth sections $s$ of $E$, and

\item if we extend $D$ to $D\,:\, {\mathcal C}^{\infty}(E\otimes \bigwedge^{k} T{ \mathcal 
F}^{\ast}) \, \longrightarrow\, {\mathcal C}^{\infty}(E\otimes 
\bigwedge^{k+1} T{ \mathcal F}^{\ast})$, then $D\circ D\,=\,0$.
\end{itemize}
Then, by Rawnsley's theorem in \cite{Ra}, the vector bundle $E$ is locally trivialized by 
sections $s$ satisfying $Ds\,=\,0$. Thus, taking such trivializations, the corresponding
transition functions are basic, and consequently, we obtain a basic vector bundle structure on $E$. 
Thus, basic vector bundle structures on $E$ correspond to flat partial ${\mathcal 
F}$-connections on $E$.

Suppose $(M,\,\mathcal F)$ is transverse holomorphic. Take the sub-bundle ${\mathcal G}\,\subset\, TM_{\C}$ such that $T{\mathcal F}_{\C}\, =\, 
{\mathcal G}\cap \overline{\mathcal G}$ and
$N^{1,0}{\mathcal F}\,=\,{\mathcal G}/({\mathcal G}\cap \overline{\mathcal G})$. Then by using 
Rawnsley's theorem, \cite{Ra}, in a similar fashion, basic holomorphic vector bundle structures on $E$
correspond to the flat partial $\overline{ \mathcal G}$-connections on $E$.

\subsection{Connections and basic Chern classes}

Let $E$ be a complex basic vector bundle over $M$.
Consider a connection operator $$\nabla\,:\, A^{\ast}_{B}(M,\,E)
\,\longrightarrow\, A^{\ast+1}_{B}(M,\,E)$$ satisfying the equation
\[\nabla(\omega s)\,=\, (d\omega )s+(-1)^{r}\omega\wedge \nabla s
\]
for $\omega\,\in\, A^{r}_{B}(M)$ and $s\,\in\,
A^{0}_{B}(M,\,E)$.
This is a usual connection operator $$\nabla\,:\, A^{\ast}(M,\,E)
\,\longrightarrow\, A^{\ast+1}(M,\,E)$$
so that the restriction of it to $T{\mathcal F}$ is the canonical flat partial ${ \mathcal F}$-connection associated with the basic vector bundle structure on $E$.
Let $$R^{\nabla}\,= \,\nabla^2\,\in\, A_{B}^{2}({M,\, \rm End}(E))$$
be the curvature of $\nabla$. For
any $1\,\le\, i\,\le \,n$, define $c_{i}(E,\nabla) \,\in\, A^{2i}_{B}(M)$ by 
\[{\rm det}\left(I- \frac{R^{\nabla}}{2\pi\sqrt{-1}}\right)
\,=\,1+\sum_{i=1}^{n}c_{i}(E,\nabla) \, .
\]
Then, as the case of usual Chern-Weil theory, the cohomology class
$$c_{i, B}(E)\, \in\, H_{B}^{2i}(M)$$ of each
$c_{i}(E,\nabla)$ is actually independent of the choice of the connection $\nabla$
taking $A^{\ast}_{B}(M,\,E)$ to $A^{\ast+1}_{B}(M,\,E)$.

\subsection{Basic Higgs bundles}

A \textit{basic Higgs bundle} over $(M,\,\mathcal F)$ is a pair
$(E, \,\theta)$ consisting of a transversely holomorphic vector bundle $E$
on $(M,\,\mathcal F)$ and a section
$$\theta\,\in\, A^{1,0}_{B}(M,\,{\rm 
End}(E))$$ satisfying the following two conditions:
$$\overline\partial_{E}\theta \,=\,0\ \ \text{ and }
\ \ \theta\wedge \theta\,=\,0\, .$$
This section $\theta$ is called a Higgs field on $E$. 
We define the canonical operator $D^{\prime\prime}\,=\,\overline\partial_{E}+\theta$.
Then the above two equations imply that $D^{\prime\prime}D^{\prime\prime}\,=\,0$.

\section{Higgs bundles on Sasakian manifolds}\label{secHig}

In this section, we discuss the results proved in \cite{BK} and also
study some fundamental examples.

\subsection{Sasakian manifolds}

Let $M$ be a $(2n+1)$-dimensional real $C^\infty$ orientable manifold. Let $TM_{\mathbb 
C}\,=\, TM\otimes_{\mathbb R}{\mathbb C}$ be its complexified tangent bundle. The Lie 
bracket operation on the locally defined vector fields on $M$ extends to a Lie bracket 
operation on the locally defined $C^\infty$ sections of $TM_{\mathbb C}$. A complex
subbundle of $TM_{\mathbb C}$ whose sections are closed under the Lie bracket
operation is called \textit{integrable}.

A {\em CR-structure} on $M$ is an 
$n$-dimensional complex sub-bundle $T^{1,0}M$ of $TM_{\mathbb C}$ such that
$T^{1,0}M\cap \overline{T^{1,0}M}
\,=\,\{0\}$ and $T^{1,0}M$ is integrable. Given such a subbundle $T^{1,0}M$, there is 
a unique sub-bundle $S$ of rank $2n$ of the real tangent bundle $TM$ together with a
vector bundle homomorphism $I\,:\,S\,\longrightarrow\, S$ satisfying the following two conditions:
\begin{enumerate}
\item $I^{2}\,=\,-{\rm Id}_{S}$, and

\item $T^{1,0}$ is the $\sqrt{-1}$--eigenbundle of $I$ acting on $S\otimes{\mathbb C}$.
\end{enumerate}
The subbundle $\overline{T^{1,0}M}$ of $TM_{\mathbb C}$ will be denoted by $T^{0,1}M$.

A $(2n+1)$-dimensional manifold $M$ equipped with a triple $(T^{1,0}M,\, S,\, I)$ as above is 
called a {\em CR-manifold}. A {\em contact CR-manifold} is a CR-manifold $M$ with a contact 
$1$-form $\eta$ on $M$ such that $\ker\eta\,=\,S$. Let $\xi$ denote the Reeb vector field for the 
contact form $\eta$. On a contact CR-manifold, the above homomorphism $I$ extends to entire 
$TM$ by setting $I(\xi)\,=\,0$.

\begin{definition}
A contact CR-manifold $(M,\, (T^{1,0}M,\, S,\, I),\, (\eta,\, \xi))$ is a {\em 
strongly pseudo-convex CR-manifold} if the Hermitian form $L_{\eta}$ on $S_x$ defined by 
$$L_{\eta}(X,\,Y)\,=\,d\eta(X,\, IY),\ \ \, X,\,Y\,\in\, S_{x},$$
is positive definite for every point $x\,\in\, M$. 
\end{definition}

We recall that given any strongly pseudo-convex CR-manifold $(M, \,(T^{1,0}M,\, S, \,I),\, (\eta,\, \xi))$, there
is a canonical Riemannian metric $g_{\eta}$ on $M$ which is defined to be
\[g_{\eta}(X,Y)\,:=\, L_{\eta}(X,Y)+\eta(X)\eta(Y)\, ,\ \ X,\,Y\,\in \,T_{x}M\, , \ x\, \in\, M\, .
\] 

\begin{definition}
A {\em Sasakian manifold} is a strongly pseudo-convex CR-manifold $$(M, \,(T^{1,0},\, S,\, I),\,
(\eta,\, \xi))$$ satisfying the condition that
$$[\xi,\, {\mathcal C}^{\infty}(T^{1,0})]\,\subset\,{\mathcal C}^{\infty}(T^{1,0}).$$
In this case, the canonical metric $g_{\eta}$ is called the Sasakian metric.
\end{definition}

For the Sasakian metric $g_{\eta}$, the Reeb vector field $\xi$ is Killing and also $\vert 
\xi\vert\,=\,1$. For a Sasakian manifold $(M, \,(T^{1,0},\, S,\, I),\, (\eta,\, \xi))$, the 
Reeb vector filed $\xi$ induces a $1$-dimensional foliation ${\mathcal F}_{\xi}$ (Reeb 
foliation). The sub-bundle ${\mathcal G}_{\xi}\,=\,T^{1,0}\oplus T{\mathcal 
F}_{\xi}\subset TM_{\C}$ produces a transversely holomorphic
structure for the foliation ${\mathcal F}_{\xi}$. The form $d\eta$ is a transversely
K\"ahler structure.

\subsection{Hodge theory}\label{subshod}

Let $(M, \,(T^{1,0}M,\, S, \,I),\, (\eta,\, \xi))$ be a compact Sasakian manifold
of dimension $2n+1$.
For the usual Hodge star operator $\ast: A^{r}(M)\to A^{2n+1-r}(M)$
associated to the Sasakian metric $g_{\eta}$, 
we define the basic Hodge star operator $$\star_{\xi}:A^{r}_{B}(M)_{\mathbb C}
\,\longrightarrow\, A^{2n-r}_{B}(M)_{\mathbb C}$$ to be
$\star_{\xi}\omega\,=\,\ast(\eta\wedge \omega)$ for $\omega\,\in\, A^{r}_{B}(M)_{\mathbb C}$.
Define the Hermitian product
\[A^{r}_{B}(M)_{\mathbb C}\times A^{r}_{B}(M)_{\mathbb C}\,\ni\,
(\alpha,\,\beta)\,\longmapsto\, \int_{M} \eta\wedge\alpha\wedge \star_{\xi}\overline{\beta}
\, \in\, {\mathbb C}\, ,
\]
and we define
\begin{itemize}
\item $\delta_{B}\,:\,A^{r}_{B}(M)\,\longrightarrow\,
A^{r-1}_{B}(M)$,

\item $\partial_{B}^{\ast}\,:\,A^{p,q}_{B}(M)\,\longrightarrow\, A^{p-1,q}_{B}(M)$,

\item $\overline\partial_{B}^{\ast}\,:\,A^{p,q}_{B}(M)\,\longrightarrow\, A^{p,q-1}_{B}(M)$ and

\item $\Lambda\,:\, A^{p,q}_{B}(M)
\,\longrightarrow\, A^{p-1,q-1}_{B}(M)$ 
\end{itemize}
to be the formal adjoints of $d\,:\,A^{r}_{B}(M)\,\longrightarrow\, A^{r+1}_{B}(M)$,\, $\partial_{B}\,:\,A^{p,q}_{B}(M)
\,\longrightarrow\, A^{p+1,q}_{B}(M)$,\, $\overline\partial_{B}\,:\,A^{p,q}_{B}(M)\,
\longrightarrow\, A^{p,q+1}_{B}(M)$ and $$L
\,:\, A^{p,q}_{B}(M)\,\ni\, x\,\longmapsto\, d\eta\wedge x\in A^{p+1,q+1}_{B}(M)$$ respectively.
Then we have the K\"ahler identities
\[[\Lambda, \,\partial_{B}]\,=\,
-\sqrt{-1}\,\overline\partial_{B}^{\ast}\ \ \text{ and } \ \
[\Lambda ,\,\overline\partial_{B}]\,=\,\sqrt{-1}\,\partial_{B}^{\ast}.
\]
These imply the Hodge decomposition
\[H^{r}_{B}(M)\,=\,
\bigoplus_{p+q=r} H^{p,q}_{B}(M)\, 
\]
(see \cite{EKA}).

\subsection{Higgs bundles}

We call a basic holomorphic vector bundle $E$ over $(M,\,{\mathcal F}_{\xi})$ a {\em 
b-holomorphic vector bundle} over the Sasakian manifold $(M, \,(T^{1,0},\, S,\, I),\, 
(\eta,\, \xi))$. By Rawnsley's theorem \cite{Ra}, b-holomorphic vector bundles $E$ correspond to 
${\mathcal C}^{\infty}$-vector bundles $E$ with a flat partial ${\mathcal 
G}_{\xi}$-connection.

We define the {\it degree} of a b-holomorphic vector bundle $E$ by 
$${\rm deg}(E)\,:=\,\int_{M}c_{1, B}(E)\wedge (d\eta)^{n-1}\wedge\eta\, .$$
A {\em b-Higgs bundle} over a Sasakian manifold $(M, \,(T^{1,0},\, S,\, I),\, (\eta,\, \xi))$ is
a basic Higgs bundle $(E,\, \theta)$ over $(M,\,{\mathcal F}_{\xi})$.

We now define the stable Higgs bundles. Denote by ${\mathcal O}_{B}$ the sheaf of basic 
holomorphic functions on $M$, and for a b-holomorphic vector bundle $E$ on $M$, denote by 
${\mathcal O}_{B}(E)$ the sheaf of basic holomorphic sections of $E$. Consider ${\mathcal 
O}_{B}(E)$ as a coherent ${\mathcal O}_{B}$-sheaf.

For a b-Higgs 
bundle $(E,\, \theta)$ over $(M, \,(T^{1,0},\, S,\, I),\, (\eta,\, \xi))$, a {\em sub-b-Higgs sheaf} of $(E,\, \theta)$ is a coherent 
${\mathcal O}_{B}$-subsheaf $\mathcal V$ of ${\mathcal O}_{B}(E)$ such that $\theta ({\mathcal V})\, \subset\, {\mathcal 
V}\otimes \Omega_{B}$, where $\Omega_{B}$ is the sheaf 
of basic holomorphic $1$-forms on $M$. By \cite[Proposition 3.21]{BH}, if ${\rm rk} 
(\mathcal V)\,<\,{\rm rk}(E)$ and the quotient ${\mathcal O}_{B}(E)/\mathcal V$ is 
torsion-free, then there is a transversely analytic sub-variety $S\,\subset\, M$ of complex 
co-dimension at least 2 such that ${\mathcal V}\big\vert_{M\setminus S}$ is given by a
basic holomorphic sub-bundle $$V\,\subset\,
E\big\vert_{M\setminus S}.$$ The degree ${\rm deg}(\mathcal V)$ can be defined by integrating $c_{1, B}(V)\wedge (d\eta)^{n-1}\wedge\eta$
on this complement $M\setminus S$.

\begin{remark}
Defining ${\rm det}(\mathcal V)\,=\,\bigwedge^{{\rm rk} (\mathcal V)} (\mathcal 
V^{\ast\ast})$ by the same way as in the complex case (\cite[Cahpter V]{Ko}), we see that 
${\rm det}(\mathcal V)$ is a b-holomorphic line bundle. However, even if $E$ admits a basic 
Hermitian metric, it is not clear whether ${\rm det}(\mathcal V)$ admits a basic Hermitian 
metric. In Corollary \ref{det2} we give a sufficient condition for ${\rm det}(\mathcal V)$ to
admit a basic Hermitian metric.
\end{remark}

\begin{definition}\label{def1}
We say that a b-Higgs bundle $(E,\, \theta)$ is {\em stable} if $E$ admits a basic Hermitian metric and for every 
sub-b-Higgs sheaf ${\mathcal V}$ of $(E,\, \theta)$ such that ${\rm rk} (\mathcal V)\,<\,{\rm 
rk}(E)$ and ${\mathcal O}_{B}(E)/\mathcal V$ is torsion-free, 
the inequality
\[\frac{{\rm deg}(\mathcal V)}{{\rm rk} (\mathcal V)}\,<\,\frac{{\rm deg}(E)}{{\rm rk}(E)}
\]
holds.

A b-Higgs bundle $(E,\, \theta)$ is called {\em polystable} if
$$
(E,\, \theta)\,=\, \bigoplus_{i=1}^k (E_i,\, \theta_i)\, ,
$$
where each $(E_i,\, \theta_i)$ is a stable b-Higgs bundle with
\[\frac{{\rm deg}( E_i)}{{\rm rk} ( E_i)}\,=\,\frac{{\rm deg}(E)}{{\rm rk}(E)}\, .
\]
\end{definition}

\begin{remark}
If a b-holomorphic vector bundle $E$ admits a basic Hermitian metric, then by the canonical 
connection (see \eqref{chco} in the
next subsection), as usual, ${\rm deg}(E)$ is a real number. But, in 
general ${\rm deg}(E)$ is not a real number (see Example \ref{line}). For the definition of 
stability, the existence of a basic Hermitian metric is very important.
\end{remark}

\begin{example}\label{exTW}
For a strongly pseudo-convex CR-manifold $(M,\, (T^{1,0},\, S,\, I),\, (\eta, \,\xi))$
there exists a unique affine connection $\nabla^{TW}$ on $TM$ such that the following
statements hold (\cite{Tan, Web}):
\begin{enumerate}
\item $\nabla^{TW}(C^{\infty}(S))\,\subset\, A^{1}(M,\,S)$, where $A^{k}(M,\,S)$ is the space of 
differential $k$-forms on $M$ with values in the vector bundle $S$.

\item $\nabla^{TW}\xi\,=\,0$, $\nabla^{TW}I\,=\,0$, $\nabla^{TW}d\eta\,=\,0$,
$\nabla^{TW}\eta\,=\,0$ and $\nabla^{TW}g_{\eta}\,=\,0$.

\item The torsion $T^{TW}$ of the affine connection $\nabla^{TW}$ satisfies the equation
\[T^{TW} (X,\,Y)\,=\, -d\eta (X,\,Y)\xi
\]
for all $X,\,Y\,\in\, S_{x}$ and $x\,\in\, M$.
\end{enumerate}
This affine connection $\nabla^{TW}$ is called the {\em Tanaka--Webster connection}. It
is known that $(M,\, 
(T^{1,0},\, S,\, I),\, (\eta, \,\xi))$ is a Sasakian manifold if and only if 
$T^{TW}(\xi,\,v)\,=\,0$ for all $v\, \in\, TM$.

For a Sasakian manifold $(M,\, (T^{1,0},\, S,\, I),\, (\eta, \,\xi))$, we consider the complex vector bundle $T^{1,0}$.
Since $\nabla^{TW}I\,=\,0$, we can define $\nabla^{TW}$ as a connection on $T^{1,0}$.
For $X\,\in\, {\mathcal C}^{\infty}(T^{1,0})$, by $T^{TW}(\xi,X)\,=\,0$ and $\nabla^{TW}\xi\,=\,0$, we have $\nabla^{TW}_{\xi}X\,=\,[\xi,\, X]$.
For $W,\,Z\,\in\, {\mathcal C}^{\infty}(T^{1,0})$, we have 
$$\nabla^{TW}_{\overline{Z}}W-\nabla^{TW}_{W}\overline{Z}-[\overline{Z}, \,W]+d\eta(\overline{Z},\,W)\xi\,=\,0\, ,
$$
and hence $\nabla^{TW}_{\overline{Z}}W\,=\,p^{1,0}([\overline{Z}, W])$, where 
$p^{1,0}\,:\,TM_{\C} \,\longrightarrow\, T^{1,0}$ is the natural projection. By these 
relations, we can say that the connection $\nabla^{TW}$ on $T^{1,0}$ induces a flat partial 
${\mathcal G}_{\xi}$-connection. Hence, there is a canonical b-holomorphic vector bundle 
structure on $T^{1,0}$. This b-holomorphic vector bundle will be called the {\em holomorphic 
tangent bundle} of the Sasakian manifold $(M,\, (T^{1,0},\, S,\, I),\, (\eta, \,\xi))$.

Regard $g_{\eta}$ as a Hermitian metric on $T^{1,0}$.
Since $\nabla^{TW}g_{\eta}\,=\,0$, this $g_{\eta}$ is actually basic and $\nabla^{TW}$ is
in fact the canonical connection corresponding to $g_{\eta}$ (see \eqref{chco}).

Define a b-Higgs bundle $(E,\, \theta)$ such that:
\begin{itemize}
\item $E\,=\,\C_{M}\oplus T^{1,0}$ where $\C_{M}$ is the trivial b-holomorphic line bundle on $M$, and
\item $\theta\,=\,\left(
\begin{array}{cc}
0&0 \\
1 & 0
\end{array}
\right)$ where $1$ means the identity element in ${\rm End} (T^{1,0})$; here basic holomorphic $1$-forms with values in $T^{1,0}$
are regarded as basic holomorphic sections of $(T^{1,0})^{\ast}\otimes T^{1,0}\,=\,{\rm End} (T^{1,0})$.
\end{itemize}
If $n\,=\,1$ and ${\rm deg}(T^{1,0})\,<\,0$, then $(E,\,\theta)$ is stable.
\end{example}

\subsection{Hermite--Einstein metric}

Let $(E, \,\theta)$ be a b-Higgs bundle over a compact Sasakian manifold $(M,\, (T^{1,0}M,
\, S, \,I),\, (\eta, \,\xi))$. 
Let $h$ be a Hermitian metric on $E$. Assume that $h$ is basic.
Define $\overline\theta_{h}\,\in\, A^{0,1}_{B}(M,\,{\rm End}(E))$ by
\begin{equation}\label{tst}
(\theta (e_{1}),\, e_{2})\,=\,(e_{1}, \,\overline\theta_{h} (e_{2}))
\end{equation}
for all $e_{1},\,e_{2}\,\in\, E_x$ and all $x\, \in\, M$.
We define the canonical (Chern) connection
\begin{equation}\label{chco}
\nabla^{h}\, :\, A^0_{B}(M,\,E)\,\longrightarrow\, A^1_{B}(M,\,E)
\end{equation}
on the b-holomorphic Hermitian bundle
$(E,\, h)$ in the following way. Take local basic holomorphic frames $e^{\alpha}_{1},\,\cdots ,\,
e_{\alpha}^{\alpha}$ of $E$ with respect to an open covering $M\,=\,\bigcup_{\alpha} U_{\alpha}$. For the Hermitian 
matrices $H_{\alpha}\,=\,(h_{i\overline{j}}^{\alpha})$ with 
$h_{i\overline{j}}^{\alpha}\,:=\,h(e^{\alpha}_{i},\, e^{\alpha}_{j})$, define
\begin{equation}\label{nah}
\nabla^{h}\,=\,d+ 
H_{\alpha}^{-1}\partial_{\xi}H_{\alpha}
\end{equation}
on each $U_{\alpha}$.

Let us consider the canonical connection
\begin{equation}\label{dh}
D^{h}\,=\,\nabla^{h}+\theta+\overline\theta_{h}
\end{equation}
on $E$. Define the operator 
$\partial_{E,h}\,:\,A^{p,q}_{B}(M,\,E)\,\longrightarrow\,
A^{p+1,q}_{B}(M,\,E)$ 
such that $$\partial_{E,h}\,=\,\partial_{\xi}+ H_{\alpha}^{-1}\partial H_{\alpha}$$ on each 
$U_{\alpha}$. Therefore, $\partial_{E,h}$ is the $(1,\, 0)$-component of $\nabla^{h}$.

The basic Hermitian metric $h$ is said to be
{\em Hermite-Einstein } if the equation
\begin{equation}\label{hyme}
\Lambda R^{D^{h} }\, =\, \lambda I
\end{equation}
holds for some constant $\lambda$.

If \eqref{hyme} holds, then $\lambda$ gets related to the degree of $E$.

\begin{example}\label{line}
Consider the ${\mathcal C}^{\infty}$--trivial complex line bundle $E\,=\,M\times \C\, 
\longrightarrow\, M$. For any $C\,\in\,\C$, we define the connection 
\begin{equation}\label{naco}
\nabla^{C}\,=\,d+C\eta
\end{equation}
on $E$, where $\eta$ as before is the contact $1$-form on the 
Sasakian manifold $M$. Then, the curvature of $\nabla^{C}$ is $Cd\eta$. Since $d\eta\,\in\, 
A^{1,1}_{B}(M)$, this $\nabla^{C}$ induces a flat partial ${\mathcal G}_{\xi}$-connection 
on $E$. Consequently, we have a non-trivial b-holomorphic vector bundle structure $E_{C}$ on $E$ 
that depends on $C\,\in\, \C$. The basic cohomology class of $-\frac{1}{2\pi\sqrt{-1}}Cd\eta$ 
is the basic first Chern class of the basic vector bundle $E_{C}$. Thus 
$\{E_{C}\}_{C\in\C}$ is a family of basic vector bundles such that $E_{C}\,\not\cong\, 
E_{C^{\prime}}$ for every $C\,\not=\,C^{\prime}$. The degree of $E_{C}$ is
$$
{\rm deg}(E_{C})\,=\,-\int_{M}\frac{1}{2\pi\sqrt{-1}}C (d\eta)^{n}\wedge \eta\, .
$$
We note that ${\rm deg}(E_{C})$ is not a real number if $C$ is not purely imaginary.
Therefore, $E_{C}$ does not admit any basic Hermitian metric if $C$ is purely imaginary.

Now take $C$ to be purely imaginary. Then the standard constant Hermitian metric $h$ on the 
${\mathcal C}^{\infty}$ trivial line bundle $E\,=\,M\times \C$ is basic on $E_{C}$. The 
connection $\nabla^{C}$ in \eqref{naco} is unitary for $h$, and hence $\nabla^{C}$ is the 
canonical connection for the basic Hermitian metric $h$. We can check that the equation
\[
\Lambda R^{\nabla^{C}}\,=\,\lambda
\]
holds.
Hence $h$ is a Hermite-Einstein metric on the Higgs line bundle $(E_{C},\,0)$.
\end{example}

\begin{remark}
For any b-holomorphic vector bundle $E$ over $(M,\, (T^{1,0}M,
\, S, \,I),\, (\eta, \,\xi))$ admitting a basic Hermitian metric, there exists a unique $C\,\in\, \C$ such that the degree of
$\widehat{E}\,=\,E\otimes E_{C}$ is $0$.
Hence, for any b-Higgs bundle $(E,\,\theta)$ admitting a basic Hermitian metric,
we can reduce the equation in\eqref{hyme} to 
\begin{equation}
\Lambda R^{D^{h} }\, =\, 0\, .
\end{equation}
\end{remark}

For a b-Higgs bundle $(E,\,\theta)$, if ${\rm rk}(E)\,=\,1$, then a basic Hermitian metric 
$h$ is a Hermite--Einstein metric on $(E,\,\theta)$ if and only if it is a 
Hermite--Einstein metric on $(E,\,0)$. For any b-holomorphic line bundle $E$, if $E$ 
admits a basic Hermitian metric $h$, then there is a basic function $f$, which is unique up 
to an additive constant, such that the rescaled metric $e^{f}h$ is a Hermite--Einstein metric 
on $(E,\,0)$. Indeed, this is an easy consequence of the K\"ahler identities; see Section 
\ref{subshod}.

\begin{theorem}[{\cite{BK}}]\label{tthbk0}
Let $(E, \,\theta)$ be a stable b-Higgs bundle over a compact Sasakian manifold $(M,\, (T^{1,0}M,
\, S, \,I),\, (\eta, \,\xi))$. Then
there exists a basic Hermite--Einstein metric $h$ on $E$ which is unique up to a positive constant.
\end{theorem}

\begin{proof}
In \cite[Theorem 5.2]{BK}, it was proved that there exists a basic Hermitian metric $h$ on $E$ such that
\[\Lambda R^{D^{h} \perp}\, =\,0\, ,
\]
where $R^{D^{h} \perp}$ is the trace-free part of the curvature $R^{D^{h}}$.
By the method as above, we can take a basic function $f$ such that the rescaled metric $e^{f}h$ is a Hermite--Einstein metric
on $(E,\,\theta)$.

We will now prove the statement on uniqueness.
Let $h$ and $h^{\prime}$ be two Hermite--Einstein metrics on $(E,\,\theta)$.
Then, there is a positive self-adjoint basic section $\sigma\, \in\, A^0(M,\, {\rm End}(E))$
such that $h^{\prime}\,=\,h\sigma$.
By the same proof as for \cite[Lemma 3.1]{Si1}, it follows that
\[\Delta_{h}^{\prime}(\sigma)\,=\,\sigma\sqrt{-1}(\Lambda R^{D^{h^{\prime}}}-\Lambda R^{D^{h}})+
\sqrt{-1}\Lambda D^{\prime\prime}(\sigma)\sigma^{-1}D^{\prime,h}(\sigma)\, ,
 \]
where $D^{\prime,h}\,=\,(\nabla^{h})^{1,0}+\overline{\theta}_{h}$ and $\Delta_{h}^{\prime}\,=\,(D^{\prime, h})^{\ast}D^{\prime, h}
\,=\,\sqrt{-1}\Lambda D^{\prime\prime}D^{\prime, h}$.
Since $h$ and $h^{\prime}$ are Hermite--Einstein, we have $\Lambda R^{D^{h^{\prime}}}-\Lambda R^{D^{h}}\,=\,0$.
Taking the trace and integrating, we now have 
\[\int_{M} \vert D^{\prime\prime}(\sigma)\sigma^{-1/2}\vert\,=\,0. \]
This implies that $D^{\prime\prime}(\sigma)\,=\,0$.
By the same arguments for stable bundles over compact K\"ahler manifolds (see \cite[Proposition 7.11, Corollary 7.12, Corollary 7.14]{Ko}),
it follows that $$\sigma \,=\,a I$$ for a non-zero $a\,\in \,\C$.
This proves the uniqueness, up to a positive constant, of the Hermite--Einstein metric $h$ on $E$.
\end{proof}

\begin{example}\label{exsl1}
Let $H^{2}$ be the hyperbolic plane. So ${\rm PSL}_{2}(\R)$ is the orientation preserving 
isometry group of $H^{2}$. Denote by $UH^{2}$ the unit tangent bundle of $H^{2}$. Then the 
action of ${\rm PSL}_{2}(\R)$ on $UH^{2}$ is simply transitive, and hence we can identify 
${\rm PSL}_{2}(\R)$ with $UH^{2}$ (see \cite{Sc}). From the isomorphism $TH^{2}\,\cong\, 
H^{2}\times \C$ we have ${\rm PSL}_{2}(\R)\,=\,UH^{2}\cong H^{2}\times S^{1}$. Consider the 
universal covering map $$p\,:\,\widetilde{\rm SL}_{2}(\R)\,\longrightarrow\, {\rm 
PSL}_{2}(\R)\, .$$ We have $\widetilde{\rm SL}_{2}(\R)\,\cong\, H^{2}\times 
\R\,\longrightarrow\, H^{2}\times S^{1}\,\cong\, {\rm PSL}_{2}(\R)$. Note that the quotient 
homomorphism $q\,:\,{\rm SL}_{2}(\R)\,\longrightarrow\, {\rm SL}_{2}(\R)/({\mathbb 
Z}/2{\mathbb Z}) \,=\, {\rm PSL}_{2}(\R)$ is a double covering. Let $\g$ be the Lie algebra 
of $\widetilde{\rm SL}_{2}(\R)$. The covering map $p_{2}\,:\, \widetilde{\rm 
SL}_{2}(\R)\,\longrightarrow\, {\rm SL}_{2}(\R)$ induces an isomorphism from $\g$ to 
$\mathfrak{sl}_{2}(\R)$. Hence, we have $\g\,=\,\langle X,\,Y, \,Z\rangle$ such that 
$[X,\,Z]\,=\,-2Y$, $[Y,\, Z] \,=\,2X$ and $[X,\,Y]\,=\,-2Z$ with
\begin{equation}\label{eqma}
X\,=\, \left(
\begin{array}{cc}
1 & 0 \\
0 & -1
\end{array}
\right),\qquad Y\,=\, \left(
\begin{array}{cc}
0 & 1 \\
1 & 0
\end{array}
\right)\qquad {\rm and} \qquad Z\,=\, \left(
\begin{array}{cc}
0 & -1 \\
1 & 0
\end{array}
\right).
\end{equation}
We regard $\g$ as the left-invariant vector fields on $\widetilde{\rm SL}_{2}(\R)$.
Then, $W\,=\,\frac{1}{2}(X-\sqrt{-1}Y)$ defines a left-invariant CR structure $T^{1,0}$ on $\widetilde{\rm SL}_{2}(\R)$.
Consider the dual $\g^{\ast}\,=\,\langle x,\,y,\, z\rangle$ as the left-invariant differential forms on $\widetilde{\rm SL}_{2}(\R)$.
Then, for the contact form $\eta\,=\,z$ with the Reeb vector field $\xi \,=\,Z$, 
we obtain a left-invariant Sasakian structure on $\widetilde{\rm SL}_{2}(\R)$.

Define the left-invariant connection $\nabla^{TW}$ such that
$$\nabla^{TW}_{X}\,=\,\nabla^{TW}_{Y}\,=\,0,\ \ \nabla^{TW}_{Z}(-)\,=\, [Z,\, -]$$
on $\g$. We can directly check that $\nabla^{TW}$ is the Tanaka-Webster connection. From the equation
$[Z,\,W]\,=\, 2\sqrt{-1}W$ we have $T_{1,0}\,\cong\, E_{2\sqrt{-1}}$, where $E_{2\sqrt{-1}}$ is the line bundle as in Example \ref{line}.

Define the Higgs bundle $(E,\,\theta)$ such that:
 \begin{itemize}
\item $E\,=\,E_{\sqrt{-1}}\oplus E_{-\sqrt{-1}}$, and

\item $\theta\,=\, w\otimes e^{\ast}_{-\sqrt{1}}\otimes e_{\sqrt{-1}}$, where $w$ is the dual of $W$ and $e_{C}$ is the
global ${\mathcal C}^{\infty}$-frame of $E_{C}$, where
$C\, \in\, \mathbb C$, satisfying the condition that $\nabla^{C}e_{C}\,=\,C\eta e_{C}$,
while $e^{\ast}_{C}$ is the dual of $e_{C}$.
\end{itemize}

Let $\widetilde{\Gamma}\,\subset\, \widetilde{\rm SL}_{2}(\R)$ be a co-compact discrete subgroup. A left-invariant Sasakian structure on 
$\widetilde{\rm SL}_{2}(\R)$ induces a Sasakian structure on the compact homogeneous space
$$M\,=\,\widetilde{\Gamma}\backslash \widetilde{\rm SL}_{2}(\R)\, .$$ 
Then regard the above defined $(E,\, \theta)$ as a b-Higgs bundle over $M$. Since ${\rm deg}(E_{\sqrt{-1}})\,<\,0$, we can check that the b-Higgs bundle
$(E,\,\theta)$ is 
stable. Consequently, by Theorem \ref{tthbk0}, there exists a basic Hermite-Einstein metric $h$ on $E$. Actually, setting $h$
to be the Hermitian metric on $E$ defined by $$h(e_{\sqrt{-1}},\,e_{\sqrt{-1}})\,=\,
h(e_{-\sqrt{-1}},\,e_{-\sqrt{-1}})\,=\,1,\ \,
h(e_{\sqrt{-1}},\ \,e_{-\sqrt{-1}})\,=\,0,$$ we have $$D^{h}\,=\,\nabla^{h}+\theta+\overline\theta_{h}\,=\,d+\left(
\begin{array}{cc}
\sqrt{-1}z&w\\
\overline{w} & -\sqrt{-1}z
\end{array}
\right),$$
and we can check that $R^{D^{h} }\,=\,0$. Thus $h$ is Hermite-Einstein.
\end{example}

\subsection{Harmonic metrics}\label{secHarm}

We consider any flat vector bundle $(E,\, D)$ over $M$ as a basic vector bundle.
Then we immediately have $$D\,:\, A^{\ast}_{B}(M,\,E)\,\longrightarrow\, A^{\ast+1}_{B}(M,\,E)\, .$$
Thus, $A^{\ast}_{B}(M,\, E)$ is a sub-complex of the de Rham 
complex $A^{\ast}(M,\,E)$ equipped with the differential $D$.
Denote by $H^{\ast}(M,\,E) $ and $H^{\ast}_{B}(M,\,E) $ the cohomology of the de Rham complex $A^{\ast}(M,\,E)$ and the
cohomology of the sub-complex $A^{\ast}_{B}(M,\, E)$ respectively. Let $h$ be a (not necessarily basic) Hermitian metric on $E$.
It gives a unique decomposition
$$
D\,=\, \nabla+\phi
$$
such that $\nabla$ is a unitary connection for $h$ and $\phi$ is a $1$-form on $M$ with values in
the self-adjoint endomorphisms of $E$ with respect to $h$. The Hermitian metric
$h$ is called {\em harmonic} if $\nabla^{\ast}\phi\,=\,0$ with
$\nabla^{\ast}$ being the formal adjoint operator of $\nabla$. It is known that the flat connection 
$(E,\, D)$ is semi-simple if and only if there exists a harmonic Hermitian metric $h$ \cite{Cor}. Recall
that $(E,\, D)$ is called semi-simple if it is a direct sum of vector
bundles with irreducible connection.
 
Assume that the Hermitian metric $h$ is basic.
This condition is equivalent to the condition that $\phi(\xi)\,=\,0$ (see \cite[Proposition 4.1]{BK}).
Then we have $$\nabla \,:\, A^{\ast}_{B}(M,\,E)
\,\longrightarrow\, A^{\ast+1}_{B}(M,\,E)\, .$$
Decompose the connection $\nabla$ as $$\nabla\, =\,\nabla^{\prime} +\nabla^{\prime\prime}\, ,$$
where $\nabla^{\prime}\,:\,A^{p,q}_{B}(M,\,E)\,\longrightarrow\,
A^{p+1,q}_{B}(M,\,E)$ and $\nabla^{\prime\prime}\,:\,A^{p,q}_{B}(M,\,E)\,\longrightarrow\, A^{p,q+1}_{B}(M,\,E)$, and also
decompose $\phi$ as 
$$\phi\,=\,\theta+\overline\theta\, ,$$ where 
$\theta\,\in\, A^{1,0}_{B}(M,\,{\rm End}(E))$ and $ 
\overline\theta\in A^{0,1}_{B}(M,\,{\rm End}(E))$.

\begin{theorem}[{\cite[Theorem 4.2]{BK}}]\label{t4.2}
Let $(M,\, (T^{1,0},\, S,\, I),\, (\eta,\, \xi))$ be a compact Sasakian manifold and
$(E, \, D)$ a flat complex vector bundle over $M$ with a Hermitian metric $h$.
Then the following two conditions are equivalent:
\begin{itemize}
\item The Hermitian structure $h$ is harmonic, i.e., $(\nabla)^{\ast}\phi\,=\,0$.

\item The Hermitian structure $h$ is basic, and the equations 
\[\nabla^{\prime\prime} \nabla^{\prime\prime} \,=\,0,\qquad 
[\theta,\,\theta]\,=\,0\qquad {\rm and} \qquad 
\nabla^{\prime\prime} \theta\,=\,0\, \]
hold.
\end{itemize}
\end{theorem}

Given a semi-simple flat vector bundle $(E,\, D)$ equipped with a harmonic metric $h$, in 
view of Theorem \ref{t4.2}, $E$ is regarded as a holomorphic vector bundle corresponding to 
$\nabla^{\prime\prime}$ and $(E,\, \theta)$ is regarded a b-Higgs bundle. We know that 
$(E,\, \theta)$ is polystable (see \cite[Section 7]{BK}), and the harmonic metric $h$ is 
actually Hermite-Einstein for $(E,\, \theta)$.

Conversely, for a Hermite-Einstein metric $h$ on a Higgs bundle $(E,\,\theta)$, using the 
Riemann bilinear relations on the basic forms, we obtain the Sasakian version of 
\cite[Proposition 3.4]{Si1} by the same proof of it; in other words, we have the inequality
\begin{equation}\label{he}
\int_{M} \left(2c_{2, B}(E) -
\frac{r-1}{r}c_{1, B}(E)^2\right)\wedge(d\eta)^{n-2}\wedge\eta \, \geq\, 0\, ,
\end{equation}
where $r\,=\,{\rm rank}(E)$, and furthermore, if the inequality in \eqref{he} is an equality,
then $$R^{D^{h} }-\frac{1}{r} {\rm Tr}R^{D^{h} }I\, =\,0.$$ If
$$c_{1,B}(E)\,=\,0\qquad {\rm and} \qquad \int_{M} c_{2,B}(E)\wedge(d\eta)^{n-2}\wedge\eta\,=\,0$$
(this implies that the inequality in \eqref{he} is an equality), then $h$ is 
in fact harmonic, i.e., $R^{D^{h} }\,=\,0$.

In view of these, we can construct polystable Higgs bundles which come from semi-simple flat vector bundles via harmonic metrics,
and we can construct flat bundles which come from polystable b-Higgs bundles $(E,\, \theta)$ for which
$$c_{1,B}(E)\,=\,0\qquad {\rm and} \qquad \int_{M} c_{2,B}(E)\wedge(d\eta)^{n-2}\wedge\eta\,=\,0$$ via Hermite-Einstein metrics.
These two constructions give an equivalence of categories between the following two:
\begin{itemize}
\item the category of polystable b-Higgs bundles $(E,\,\theta)$ satisfying
$$c_{1,B}(E)\,=\,0\qquad {\rm and} \qquad \int_{M} c_{2,B}(E)\wedge(d\eta)^{n-2}\wedge \eta\,=\,0$$
over a compact Sasakian manifold $(M,\, (T^{1,0}M, \, S, \,I),\, (\eta, \,\xi))$, and

\item the category of semi-simple flat vector bundles over $M$.
\end{itemize}
See \cite[Section 7]{BK} and \cite[Corollary 1.3]{Si1}.

\begin{remark}[{Cf. \cite{DonI, UY}}]\label{KHS}
We say that a b-holomorphic vector bundle $E$ over a compact Sasakian manifold $(M,\, (T^{1,0}M,
\, S, \,I),\, (\eta, \,\xi))$ is stable (respectively, polystable) if the b-Higgs bundle $(E,\,0)$ is stable
(respectively, polystable). Restricting the above equivalence of categories, we have an equivalence between
the following two categories:
\begin{itemize}
\item the category of polystable b-holomorphic bundles $E$ satisfying
$$c_{1,B}(E)\,=\,0\qquad {\rm and} \qquad \int_{M} c_{2,B}(E)\wedge(d\eta)^{n-2}\wedge\eta\,=\,0$$
over a compact Sasakian manifold $(M,\, (T^{1,0}M, \, S, \,I),\, (\eta, \,\xi))$, and

\item the category of unitary flat vector bundles over $M$.
\end{itemize}
\end{remark}

\begin{example}\label{Exsl22}
We now return to Example \ref{exsl1}.
The differential of the covering map $p_{2}\,:\, \widetilde{\rm SL}_{2}(\R)\,\longrightarrow\, {\rm SL}_{2}(\R)$
is identified with the $\mathfrak{sl}_{2}(\R)$-valued left-invariant $1$-form $$\omega \,=\, x\otimes X+ y\otimes Y +z\otimes Z
\,\in\, \g^{\ast}\otimes \mathfrak{sl}_{2}(\R)\, ,$$ where $X$, $Y$ and $Z$ are defined in \eqref{eqma}.
For the ${\mathcal C}^{\infty}$--trivial vector bundle $E\,=\,M\times \C^{2}$ over
$M\,=\,\widetilde{\Gamma}\backslash \widetilde{\rm SL}_{2}(\R)$, we define the connection
$D\,=\,d+\omega$. Then the monodromy homomorphism for the flat bundle $(E,\,D)$
$$
\pi_{1}(M,\, I) \,=\, \widetilde{\Gamma}\, \longrightarrow\, {\rm SL}_{2}(\R)
$$
coincides with the restriction of the above projection
$p_{2}\,:\, \widetilde{\rm SL}_{2}(\R)\,\longrightarrow\, {\rm SL}_{2}(\R)$
to the subgroup $\widetilde{\Gamma}\, \subset\, \widetilde{\rm SL}_{2}(\R)$.
Let $h$ be the Hermitian metric on $E\,=\,M\times \C^{2}$ given by the standard Hermitian structure on $\C^{2}$.
If $D\,=\,\nabla +\phi$ is the canonical decomposition, then $\nabla \,=\,d +z\otimes Z$ and
$\phi\,=\, x\otimes X+ y\otimes Y$. It is straight-forward to check that $h$ is
actually harmonic. This metric coincides with the Hermite-Einstein
metric in Example \ref{exsl1} up to the conjugation by an invertible matrix.
\end{example}

\begin{example}
Define the b-Higgs bundle $(E,\,\theta)$, where $E\,=\,\C_{M}\oplus T^{1,0}$ is as in Example \ref{exTW}.
If $(E,\,\theta)$ is stable, then there exists a basic Hermite-Einstein metric $h$ on $E$.
In this case, we have the inequality
\begin{equation}\label{c1}
\int_{M} \left(2c_{2, B}(T^{1,0}) -
\frac{n}{n+1}c_{1, B}(T^{1,0})^2\right)\wedge(d\eta)^{n-2}\wedge\eta \, \geq\, 0\, .
\end{equation}
This Miyaoka-Yau type inequality is also proved in \cite{Za} for another assumption by using Sasaki-Einstein metrics. If in \eqref{c1}
the equality holds, then
\begin{equation}\label{dh2}
R^{D^{h} }-\frac{1}{r} {\rm Tr}R^{D^{h} }I \,=\,0,
\end{equation}
where $D^h$ is the connection defined as in \eqref{dh}. From \eqref{dh2} it follows
that the connection $D^{h}$ is projectively flat. Moreover, if $$2\pi 
\sqrt{-1}c_{1, B}(T^{1,0})\,=\,(n+1)C[d\eta]$$ for some $C\,\in\, \C$, then we have $c_{1,B}(E\otimes E_{C})\,=\,0$,
i.e., the projective representation of 
$\pi_{1}(M,\, I)$ given by the monodromy of the projectively flat bundle $(E,\, D^{h})$
gets lifted to a linear representation of $\pi_{1}(M,\, I)$.
\end{example}

\section{Quasi-regular Sasakian manifolds and orbifolds}
\subsection{Quasi-regularity}\label{q-rv}

Recall from \cite{BoG} that a compact Sasakian manifold $$(M,\, (T^{1,0}, \,S,\, I),\, (\eta, \,\xi))$$ is called {\em quasi-regular} if every 
leaf of the foliation $\mathcal F_{\xi}$ is closed. For any given compact Sasakian manifold $(M,\, (T^{1,0}, \,S,\, I),\, (\eta, \,\xi))$, we can 
take another contact form $\eta^{\prime}$ with the Reeb vector field $\xi^{\prime}$ so that $(M,\, (T^{1,0}, \,S,\, I),\, (\eta^{\prime}, 
\,\xi^{\prime}))$ is quasi-regular (see \cite[Section 8.2.3]{BoG}). More precisely, we can take a Killing vector field $\chi$ commuting with 
$\xi$ such that $\xi^{\prime}\,=\,\xi+\chi$ and $\eta^{\prime}\,=\,\frac{\eta}{1+\eta(\chi)}$.

Let $(M,\, (T^{1,0}, \,S,\, I),\, (\eta, \,\xi))$ be a compact Sasakian manifold.
Take the flow on $M$
$$\varphi\,:\, \R\times M\, \longrightarrow\, M,\ \ \,
(t,\,x)\,\longmapsto\, \varphi_{t}(x),$$ generated by the Reeb vector field $\xi$.
This means that $\frac{d}{dt }\varphi_{t}(x)\vert _{t=0}\,=\,\xi_{x}$.
Then each orbit of this flow is a leaf of the foliation $\mathcal F_{\xi}$.
Let $E$ be a basic vector bundle over the foliated manifold $(M,\,{\mathcal F}_{\xi})$.
We define a natural action
\begin{equation}\label{p1}
\Phi^{E}\,:\,\R\times E\,\longrightarrow\, E,\ \ \, (t,\,e)\,\longmapsto\, \Phi^{E}_{t}(e)
\end{equation}
as follows:
For small $s\,>\,0$, and any $x\,\in\, M$, take the local trivialization $E_{\vert U}\,\cong\, U\times \C^{{\rm rank}(E)}$
given by basic sections on a neighborhood $U$ of $x$ so that $\varphi_{t}(x)\,\in\, U$ for
any $t\,\in\, [0,\,s]$.
So $\Phi^{E}_{t}(x,\,v)\,=\,(\varphi_{t}(x), \,v)$ for $(t,\, x,\,v)\,\in\, [0,\,s]\times E_{\vert U}$.
Since, the transition functions are constant on each orbit of the flow $\varphi$, this is well-defined and 
we can extend to $$\Phi^{E}_{t}\,:\, E\,\,\longrightarrow\, E$$ for all $t\, \in\, \mathbb R$.
For the flat partial ${\mathcal F}_{\xi}$-connection $\nabla$ associated with the basic 
vector bundle $E$, the fiber vector $\Phi^{E}_{t}(e)$ is the $\nabla$-parallel transport of 
$e\,\in\, E_{x}$ along the integral curve $[0,\,t]\,\ni\, s\,\longmapsto\, 
\varphi_{s}(x)\,\in\, M$. The map $\Phi^{E}$ in \eqref{p1} sends any
$(t,\, e)\, \in\, \R\times E$ to $\Phi^{E}_{t}(e)$.

For the natural projection $p_{E}\,:\,E\,\longrightarrow\, M$, we have $p_{E}\circ \Phi^{E}_{t}\,=\,\varphi_{t}$.
For a section $s\,\in\, {\mathcal C}^{\infty}(E)$, for each $t\,\in \,\R$, we define $(\Phi^{E}_{t})^{\ast}(s)\,\in\,
{\mathcal C}^{\infty}(E)$ by $(\Phi^{E}_{t})^{\ast}(s)(x)\,=\,(\Phi^{E}_{t})^{-1}s(\varphi_{t}(x))$.
A section $s$ is called $\Phi^{E}$-invariant if $(\Phi^{E}_{t})^{\ast}(s)\,=\,s$ for every $t\,\in\, \R$. 
In particular, a Hermitian metric on $E$ is called $\Phi^{E}$-invariant if
$h(\Phi^{E}_{t}(e_{1}),\, \Phi^{E}_{t}(e_{2}))\,=\,h(e_{1},\, e_{2})$ for every $t\,\in\, \R$ and all $e_{1},\, e_{2}\,\in\, E_x$
for each $x\, \in\, M$.

\begin{lemma}
A section $s$ is $\Phi^{E}$-invariant if and only if $s$ is basic.
\end{lemma}

\begin{proof}
For the flat partial ${\mathcal F}_{\xi}$-connection $\nabla$ associated with the basic vector bundle $E$,
$$(\Phi^{E}_{t_{0}})^{\ast}\nabla_{\xi} s\,=\, (\Phi^{E}_{t_{0}})^{\ast}\frac{d}{dt }_{\vert_{t=0}}( \Phi^{E}_{t})^{\ast} s
\,=\,\frac{d}{dt}_{\vert_{t=t_{0}} }( \Phi^{E}_{t})^{\ast} s$$ for all $t_{0}\,\in\, \R$.
If $s$ is $\Phi^{E}$-invariant, then we have $\nabla_{\xi} s\,=\,0$, and hence $s$ is basic.

If $s$ is basic, then $\nabla_{\xi} s\,=\,0$ and hence $\frac{d}{dt}_{\vert_{t=t_{0}} }( \Phi^{E}_{t})^{\ast} s\,=\,0$.
Since $ \Phi^{E}_{0}\,=\,{\rm Id}_E$, the uniqueness of a solution of an
ordinary differential equation implies that $(\Phi^{E}_{t})^{\ast} s\,=\,s$.
\end{proof}

If $(M,\, (T^{1,0}, \,S,\, I),\, (\eta, \,\xi))$ is quasi-regular, then by Wadsley's 
Theorem in \cite{Wa}, the flow $\varphi\,:\, \R\times M\,\longrightarrow\, M$ induces a 
smooth action $\psi\,:\, S^{1}\times M\,\longrightarrow\, M$, in other words, there exists 
a positive $r\,\in\, \R$ such that $\varphi_{r}(x)\,=\,x$ for all $x\,\in\, M$. The minimum 
of all such $r$ is called the {\em period} of $(M,\, (T^{1,0}, \,S,\, I),\, (\eta, \,\xi))$.

\begin{remark}\label{Wad}
The following more detailed comments on Wadsley's result will be useful later.
Let $\Lambda\,:\, M\,\longrightarrow\, {\R}^{> 0}$ be the unique
positive valued function satisfying the conditions that for all $x\, \in\, M$,
\begin{itemize}
\item $\varphi_{\Lambda(x)}(x)\,=\,x$, and

\item $\varphi_{t}(x)\,\not=\, x$ for all $0\,<\,t\,<\,\Lambda(x)$.
\end{itemize}
Then $\Lambda$ is lower semi-continuous function.
The subset $M_{reg}\,=\, \{x\,\in\, M\,\,\mid\,\, \Lambda(x)=r\}$ is open dense
in $M$, where $r$ is the period of a compact
quasi-regular Sasakian manifold $(M,\, (T^{1,0}, \,S,\, I),\, (\eta, \,\xi))$.
\end{remark}

\begin{definition}\label{DQRR}
Given a compact quasi-regular Sasakian manifold $(M,\, (T^{1,0}, \,S,\, I),\, (\eta, 
\,\xi))$ with the period $r$, a basic vector bundle $E$ over the foliated manifold 
$(M,{\mathcal F}_{\xi})$ is called {\em quasi-regular} if
$$\Phi^{E}\,:\, \R\times E\, \longrightarrow\, E\, , \, \ \ (t,\, e)\, \longmapsto\,
\Phi^{E}_{t}(e)$$
induces a smooth action $\Psi^{E}\,:\, S^{1}\times E\,\longrightarrow\, E$, or in other 
words, there exists a positive integer $m$ such that $\Phi_{mr}^{E}(e)\,=\,e$ for all $e\, 
\in\, E$. A quasi-regular basic vector bundle $E$ is called {\em regular} if the above 
integer $m$ can be taken to be $1$.
\end{definition}

It is evident that a basic sub-bundle of a quasi-regular (respectively, regular) vector 
bundle is also quasi-regular (respectively, regular).

\begin{remark}
The notion of quasi-regularity (respectively, regularity) in Definition \ref{DQRR} is a 
geometric refinement of the notion of being ``virtually-basic'' (respectively, ``basic'') 
(see \cite[Definition 4.2]{BM}). It may be clarified that this notion of ``basic''-ness is 
different from the basicness of vector bundles over $(M,\, {\mathcal F}_{\xi})$. In the 
arguments in \cite{BM}, a structure of (singular) projective variety of the orbit space for 
the action $\psi\,:\, S^{1}\times M\,\longrightarrow\, M$, and its desingularization, are 
used; however, we do not need them here. In fact, these considerations have certain 
disadvantages. On the orbit space of the action $\psi\,:\, S^{1}\times M\,\longrightarrow\, 
M$ the complex orbifold structure explained in Section \ref{Orbsec} is more explicit and 
more appropriate than the structure of the projective variety associated to a quasi-regular 
Sasakian manifold. There is no canonical choice of a desingularization of it. It seems that 
the correspondence in \cite[Section 5]{BM} depends on the choice of a desingularization. 
Our correspondence given in Section \ref{secHig} evidently depends only on the Sasakian 
structure.
\end{remark}

\begin{lemma}\label{lemSinv}
Let $(M,\, (T^{1,0}, \,S,\, I),\, (\eta, \,\xi))$ be a compact quasi-regular Sasakian manifold.
If a basic vector bundle $E$ over the foliated manifold $(M,\,{\mathcal F}_{\xi})$ is
quasi-regular, then there exists a $\Phi^{E}$-invariant Hermitian metric on $E$, in
particular, there is a basic Hermitian metric.
\end{lemma}

\begin{proof}
Take a bi-invariant measure $d\mu $ on $S^{1}$.
For an arbitrary Hermitian metric $h$, define $\widetilde{h}$ by
\[\widetilde{h}(e_{1},\,e_{2})\,=\, \int_{g\in S^{1}}h(\Psi^{E}_{g}(e_{1}),\Psi^{E}_{g}(e_{2}))d\mu
\]
for $e_{1},\, e_{2}\,\in\, E_{x}$ and $x\, \in\, M$.
Then $\widetilde{h}$ is a $\Phi^{E}$-invariant Hermitian metric on $E$.
\end{proof}

\begin{lemma}\label{det}
Let $E$ be a b-holomorphic vector bundle over a compact quasi-regular Sasakian manifold 
$(M,\, (T^{1,0}, \,S,\, I),\, (\eta, \,\xi))$. Assume that $E$ is a quasi-regular 
(respectively, regular) basic vector bundle over $(M,\,{\mathcal F}_{\xi})$. Then, for any 
coherent ${\mathcal O}_{B}$-subsheaf $\mathcal V$ of ${\mathcal O}_{B}(E)$, the 
b-holomorphic line bundle ${\rm det}(\mathcal V)$ is also quasi-regular (respectively, 
regular).
\end{lemma}

\begin{proof}
Just as in the complex case, there is a transversely analytic sub-variety $S\,\subset\, M$ of complex 
co-dimension at least $2$ such that ${\rm det}(\mathcal V)_{\vert M\setminus S}$ is a sub-bundle of $\bigwedge^{{\rm rk} (\mathcal V)} E$.
Hence, we have $\Phi_{mr}^{E}\,=\,{\rm Id}$ on the complement $M\setminus S$.
Since $M\setminus S$ is dense in $M$, we conclude that $\Phi_{mr}^{E}\,=\,{\rm Id}_E$ on $M$.
Hence the lemma follows.
\end{proof}

Lemma \ref{det} and Lemma \ref{lemSinv} together give the following:

\begin{corollary}\label{det2}
Let $E$ be a b-holomorphic vector bundle over a compact quasi-regular Sasakian manifold 
$(M,\, (T^{1,0}, \,S,\, I),\, (\eta, \,\xi))$. Assume that $E$ is a quasi-regular
basic vector bundle over $(M,\,{\mathcal F}_{\xi})$. Then, for any
coherent ${\mathcal O}_{B}$-subsheaf $\mathcal V$ of ${\mathcal O}_{B}(E)$, the
b-holomorphic line bundle ${\rm det}(\mathcal V)$ admits a basic Hermitian metric.
\end{corollary}

\begin{example}
The b-holomorphic vector bundle $T^{1,0}$ is regular.
This follows easily from the lower semi-continuity of the function $\Lambda\,:\, M\,
\longrightarrow\, \R$ in Remark \ref{Wad}.
\end{example}

We say that a b-Higgs bundle $(E,\,\theta)$ over a compact quasi-regular Sasakian manifold $(M,\, (T^{1,0}, \,S,\, I),\,
(\eta, \,\xi))$ is quasi-regular (respectively, regular) if the basic vector bundle $E$ over $(M,\,{\mathcal F}_{\xi})$ is
quasi-regular (respectively, regular).

\begin{example}\label{PSL2g}
Let $\Gamma$ be a discrete subgroup in ${\rm PSL}_{2}(\R)$ such that the quotient $\Gamma\backslash{\rm PSL}_{2}(\R)$ is compact.
So $\Gamma$ is a Fuchsian subgroup without parabolic elements, and it is generated by 
\[a_{1},\,\cdots ,\,a_{g},\, b_{1},\,\cdots ,\,b_{g},\, x_{1},\,\cdots ,\,x_{m}
\]
with the relations
\begin{equation}\label{dm}
[a_{1},\,b_{1}]\dots [a_{g},\,b_{g}]x_{1}\dots x_{m}\,=\,I
\end{equation}
and $x_{j}^{p_j}\,=\,I$ for some positive integers $p_{j}\,>\,1$. Set $M\,=\,\Gamma\backslash {\rm PSL}_{2}(\R)$.
Let $p\,:\,\widetilde{\rm SL}_{2}(\R)\,\longrightarrow\, {\rm PSL}_{2}(\R)$ be the universal
covering, and define $\widetilde{\Gamma}\,:=\,p^{-1}(\Gamma)$.
Therefore, we have $M\,=\,\widetilde{\Gamma}\backslash \widetilde{\rm SL}_{2}(\R)$.
This $M$ is equipped with a Sasakian structure described in Example \ref{exsl1}.

Consider the universal covering map $p_{2}\,:\, \widetilde{\rm SL}_{2}(\R)\,\longrightarrow\, {\rm SL}_{2}(\R)$,
and define $\Gamma_{2}\,:=\,p_{2}(\widetilde{\Gamma})$.
Then $\Gamma_{2}\,=\, q^{-1}(\Gamma)$ for the quotient map $q\,:\,{\rm SL}_{2}(\R)\,\longrightarrow\,
{\rm PSL}_{2}(\R)$. We have $M\,=\,\Gamma_{2}\backslash {\rm SL}_{2}(\R)$.
The Sasakian manifold $(M,\, (T^{1,0},
\, S, \,I),\, (\eta, \,\xi))$ is quasi-regular, and the map $\varphi_{t}$ coincides with the right translation
action of $ \left(
\begin{array}{cc}
\cos t & -\sin t \\
\sin t & \cos t
\end{array}
\right)$.
Considering $M$ as $\Gamma\backslash {\rm PSL}_{2}(\R)$, the action $\psi\,:\, S^{1}\times 
M\,\longrightarrow\, M$ induced by $\{\varphi\}$ coincides with the right translation action of 
$S^{1}\,=\,q({\rm SO}(2))$ on ${\rm PSL}_{2}(\R)$, and we have $\varphi_{\pi}\,=\,{\rm Id}$. 
Since the left translation action of $\Gamma$ on $H\, :=\, {\rm PSL}_{2}(\R)/S^{1}$ is 
effective and properly discontinuous and has a fundamental polygon, we conclude that the 
period $r$ of the quasi-regular Sasakian manifold $(M,\, (T^{1,0}, \, S, \,I),\, (\eta, 
\,\xi))$ is $r\,=\,\pi$ and $\Gamma g\,\in\, M_{reg}$ if and only if $gS^{1}\,\in\, H$ is 
not fixed by any element in $\Gamma-\{{\rm I}\}$. If $\Gamma$ has no elliptic element, 
i.e., $m\,=\,0$ (see \eqref{dm}), then the Sasakian manifold $(M,\, (T^{1,0}M, \, S, 
\,I),\, (\eta, \,\xi))$ is actually regular.

We have $T^{1,0}\,\cong\, E_{2\sqrt{-1}}$, and hence the line bundle $E_{2\sqrt{-1}}$ is 
regular. The line bundle $E_{\sqrt{-1}}$ is quasi-regular but not regular, because we have 
$\Phi^{E_{\sqrt{-1}}}_{\pi}\,=\,-{\rm Id}$. Consider the b-Higgs bundle 
\begin{equation}\label{et}
(E\,=\,E_{\sqrt{-1}}\oplus E_{-\sqrt{-1}},\,\,\theta\,=\,w\otimes 
e^{\ast}_{\sqrt{-1}}\otimes e_{\sqrt{-1}})
\end{equation}
as in Example \ref{exsl1}. This b-Higgs bundle $(E,\, \theta)$
is quasi-regular. The Hermite-Einstein metric $h$ for this b-Higgs bundle is in fact 
$\Phi^{E}$-invariant. This metric $h$ is harmonic. As we saw in Example \ref{Exsl22}, the 
corresponding flat bundle $(E,\,D^{h})$ is isomorphic to the flat bundle whose
monodromy representation $\widetilde{\Gamma}\,\longrightarrow\, {\rm SL}_{2}(\C)$
is the homomorphism $\widetilde{\Gamma}\,\stackrel{p_2}{\longrightarrow}\,\Gamma_{2}
\,\subset\, {\rm SL}_{2}(\C)$.
 
Assume that $\Gamma$ contains no element of order $2$. Then we have an isomorphism 
$$\Gamma_{2}\,\cong \,\Gamma\times ({\Z}/2{\Z})$$ (see the proof of \cite[Theorem 1]{Pa}). 
Let $$\tau\,:\, \widetilde{\Gamma}\,\longrightarrow\, {\rm GL}_{1}(\C)$$ be the composition 
of the homomorphism $\widetilde{\Gamma}\,\stackrel{p_2}{\longrightarrow}\, \Gamma_{2} 
\,\cong\, \Gamma\times (\Z/2\Z)$ and the natural projection $$\Gamma\times \Z/2\Z\,\longrightarrow\, 
{\Z}/2\Z \,=\, \{\pm 1\}\,\subset\, {\rm GL}_{1}(\C)\, .$$ Consider the flat line bundle 
$E^{\tau}$ over $M\,=\,\widetilde{\Gamma}\backslash \widetilde{\rm SL}_{2}(\R)$ given by 
this homomorphism $\tau$. Then we have $\Phi^{E^{\tau}}_{\pi}\,=\,-{\rm Id}$. Thus, 
$E^{\tau}_{\sqrt{-1}}\,=\,E^{\tau}\otimes E_{\sqrt{-1}}$ is regular. For the b-Higgs bundle 
$(E,\,\theta)$ in \eqref{et}, the vector
bundle $E^{\tau}\otimes E$ is regular. Since $E^{\tau}$ is unitary flat, 
the Hermite-Einstein metric $h^{\prime}$ for $(E^{\tau}\otimes E, \,\theta)$ is also 
harmonic. The corresponding flat bundle $(E^{\tau}\otimes E,\, D^{h^{\prime}})$ is isomorphic to the flat 
bundle whose monodromy representation $\widetilde{\Gamma}\,\longrightarrow\, {\rm 
SL}_{2}(\C)$ is the composition of the homomorphism $\widetilde{\Gamma}\, 
\stackrel{p_2}{\longrightarrow}\, \Gamma_{2} \,\cong\, \Gamma\times (\Z/2\Z)$ with the 
projection $\Gamma\times (\Z/2\Z)\,\longrightarrow \,\Gamma\, \subset\, {\rm SL}_{2}(\C)$ to 
the first factor.

\begin{remark}
The isomorphisms $\Gamma_{2}\,\cong\, \Gamma\times (\Z/2\Z)$ are classified by $H^{1}(\Gamma,\, \Z/2\Z)$.
If $\Gamma$ has no elliptic element, then $H^{1}(\Gamma,\, \Z/2\Z)\cong (\Z/2\Z)^{2g}$.
\end{remark}
\end{example}

\subsection{Orbifolds}\label{Orbsec}

An $n$-dimensional complex {\em orbifold} is a paracompact Hausdorff space $X$ with a family ${\mathcal U}\,=\,\{(U_{\alpha},\,
\widetilde{U}_{\alpha},\,\Gamma_{\alpha},\, \phi_{\alpha})\}_{\alpha}$ such that:
\begin{itemize}
\item For each $\alpha$,
\begin{equation}\label{ual}
\widetilde{U}_{\alpha}\, \subset\, \C^n
\end{equation}
is a connected open subset containing the origin, $\Gamma_{\alpha}$ is a
finite subgroup in the unitary group ${\rm U}(n)$ and $\phi_{\alpha}$ is a $\Gamma_{\alpha}$-invariant continuous map from
$\widetilde{U}_{\alpha}$ into an open subset $U_{\alpha}$ in $M$ which induces a homeomorphism between $\widetilde{U}_{\alpha}/\Gamma_{\alpha}$
and $U_{\alpha}$.

\item $M\, =\, \bigcup_{\alpha} U_{\alpha}$.

\item For any two $(U_{\alpha},\,\widetilde{U}_{\alpha},\, \Gamma_{\alpha},\, \phi_{\alpha})$ and
$(U_{\beta}, \,\widetilde{U}_{\beta},\,\Gamma_{\beta},\, \phi_{\beta})$ with $U_{\alpha}\cap U_{\beta}\,\not=\,\emptyset$,
for every $x\,\in\, U_{\alpha}\cap U_{\beta}$ there exists
$(U_{\gamma},\, \widetilde{U}_{\gamma},\, \Gamma_{\gamma},\, \phi_{\gamma})$ satisfying the conditions that $x\,\in\, U_{\gamma}$ and there are
holomorphic embeddings $$\lambda_{\alpha\gamma}\,:\,\widetilde{U}_{\gamma}\,\longrightarrow\, \widetilde{U}_{\alpha}\ \ \text{ and }\ \
\lambda_{\beta\gamma}\,:\,\widetilde{U}_{\gamma}\,\longrightarrow\, \widetilde{U}_{\beta}$$
such that $\phi_{\alpha}\circ \lambda_{\alpha\gamma}\,=\,\phi_{\gamma}$ and $\phi_{\beta}\circ \lambda_{\beta\gamma}\,=\,\phi_{\gamma}$.
\end{itemize}

If every finite group $\Gamma_{\alpha}$ in the above definition is trivial, then the 
orbifold $(X,\,{\mathcal U})$ is a complex manifold of dimension $n$. A Hermitian metric on 
a complex orbifold $(X,\,{\mathcal U})$ is a family $\{g_{\alpha}\}_{\alpha}$ consisting of 
a Hermitian metric $g_\alpha$ on each $\widetilde{U}_{\alpha}$ which is 
$\Gamma_{\alpha}$-invariant and all $\lambda_{\alpha\gamma}$ are isometries. A Hermitian 
metric on a complex orbifold $(X,\,{\mathcal U})$ is K\"ahler if each $g_{\alpha}$ is a 
K\"ahler metric on $\widetilde{U}_{\alpha}$. We say that a complex orbifold $(X,\,{\mathcal 
U})$ is {\em locally cyclic} if all the finite groups $\Gamma_{\alpha}$ are cyclic.

A differential form on $(X,\,{\mathcal U})$ is a family $\{\omega_{\alpha}\in 
A^{\ast}(\widetilde{U}_{\alpha})\}_{\alpha}$ such that each $\omega_{\alpha}$ is 
$\Gamma_{\alpha}$-invariant and they are compatible with the transition maps 
$\lambda_{\alpha\gamma}$. The de Rham complex $(A^{\ast}(X),\, d)$ and the Dolbeault 
complex $(A^{\ast,\ast}(X),\, \partial,\,\overline\partial)$ of a complex orbifold 
$(X,\,{\mathcal U})$ are defined in the standard manner. As usual, de Rham's theorem says 
that the cohomology of the de Rham complex $(A^{\ast}(X),\, d)$ of a complex orbifold 
$(X,\, {\mathcal U})$ is isomorphic to the cohomology $H^{\ast}(X,\,\R)$ of the underlying 
topological space (see \cite{Sat}).

Let $(M,\, (T^{1,0}, \,S,\, I),\, (\eta, \,\xi))$ be a compact quasi-regular Sasakian manifold with the period $r$. Consider the circle action 
$\psi\,:\, S^{1}\times M\,\longrightarrow\, M$ induced by the flow $\varphi\,:\, \R\times M \,\longrightarrow\, M$ generated by the Reeb
vector field $\xi$. By the slice theorem (cf. \cite[Theorem L2.1]{Au}), 
the quotient $X\,=\, M/S^{1}$ admits a canonical complex orbifold structure $(X,\, {\mathcal U})$. To
explain this with more details, we take ${\mathcal U}$ to be consisting 
of $\{(U_{x},\, \widetilde{U}_{x},\, \Gamma_{x},\, \phi_{x})\}_{x\in X}$, where $U_{x}$ is the quotient of a $S^{1}$-invariant neighborhood of 
$x$, $\widetilde{U}_{x}$ is an open neighborhood of $0\,\in\, T^{1,0}_{x}$ and $\Gamma_{x}$ is the finite group generated by the linear map 
$\Phi^{T^{1,0}}_{\Lambda(x)}$ on $T^{1,0}_{x}$ (the notations $\Phi$ and $\Lambda(x)$ were introduced in Section \ref{q-rv}). Recall that 
$\Gamma_{x}$ acts complex linearly on $T^{1,0}_{x}$. Note that $M_{reg}$ is connected. Since 
$\Gamma_{x}$ is a subgroup in $S^{1}$, this complex orbifold $(X, \,{\mathcal U})$ is locally cyclic. Let
$$X_{reg}\,=\,\{S^{1}x\,\in\, X\,\mid\, x\,\in \,M_{reg}\}.$$ Then $X_{reg}$ 
is a connected complex manifold. It is straight-forward to check that the transversely K\"ahler structure $d\eta$ induces a
K\"ahler metric on the orbifold $X$.

\begin{remark}[{\cite{ALR, BoG}}]\label{Orbifun}
It is known that any orbifold $X$ can be represented as the quotient of a manifold $M$ by a locally free action of a compact Lie group $K$.
Let $E_K\, \longrightarrow\, B_K$ be the universal principal $K$--bundle.
Define the topological space $BX\,:=\, E_K\times_{K} M$.
The homotopy type of $BX$ depends only on the orbifold structure on $X$.
We define the orbifold cohomology $H^{k}_{orb}(X,\,R)$, with coefficients in a ring $R$, and the orbifold homotopy group $\pi_{k}^{orb}(X,\,x)$
to be $$H^{k}_{orb}(X,\, R)\,=\, H^{k}(BX, R)\qquad{\rm and} \qquad \pi_{k}^{orb}(X,\,x)\,=\, \pi_{k}(BX,\, \widetilde{x})$$
respectively, where $\widetilde{x}\,\in\, BX$ is a lift of $x$ for the natural projection $BX\,\longrightarrow\, X$. The above
defined $\pi_{1}^{orb}(X,\, x)$ is called the orbifold fundamental group of $X$.
In general, $H^{k}_{orb}(X,\, \Z)$ is not isomorphic to the cohomology $H^{k}(X,\, \Z)$ of the underlying topological space.
On the other hand, $H^{k}_{orb}(X,\,\Q)$ is naturally isomorphic to $H^{k}(X,\,\Q)$.

Let $(X,\,{\mathcal U})$ be a complex orbifold with a K\"ahler metric $\{g_{\alpha}\}_{\alpha}$.
By Satake's orbifold de Rham theorem, $\{g_{\alpha}\}_{\alpha}$ gives a cohomology class $[\omega]\,\in\, H^{2}(X,\, \R)$.
We say that $\{g_{\alpha}\}_{\alpha}$ is a Hodge metric if the cohomology class $[\omega]$ lies in the image of the homomorphism
$$H^{2}_{orb}(X,\,\Z)\,\longrightarrow\, H^{2}_{orb}(X,\,\R)\,\cong\, H^{2}(X,\,\R).$$
It is known that for a compact quasi-regular Sasakian manifold
$(M,\, (T^{1,0}, \,S,\, I),\, (\eta, \,\xi))$ 
the K\"ahler metric on the complex orbifold $X\,=\,M/S^{1}$
induced by $d\eta$ is a Hodge metric, and conversely, for
a locally cyclic complex orbifold $X$ with a Hodge metric $\{g_{\alpha}\}_{\alpha}$, there exists a quasi-regular Sasakian manifold
$(M,\, (T^{1,0}, \,S,\, I),\, (\eta, \,\xi))$ such that $X\,=\,M/S^{1}$ and $\{g_{\alpha}\}_{\alpha}$ is induced by $d\eta$ 
(see \cite[Chapter 7]{BoG} and also \cite[Theorem 4.47, 4.50]{RC}).

By an analog of the Kodaira embedding theorem, a complex orbifold $X$ with a Hodge metric is a projective variety \cite{Bai}.
It may be mentioned that the orbifold fundamental group is an invariant for orbifolds but not so for algebraic varieties
(see the proof of \cite[Lemma 4.5.6]{BoG}).
\end{remark}

For $x\,\in\, M_{reg}$, we consider the induced homomorphism $\psi(x)_{\ast}\,:\,\pi_{1}(S^1,\,1)\,\longrightarrow\, \pi_{1}(M, \,x)$
in terms of orbifolds.
We know that the quotient map $M\,\longrightarrow\, X$ can be lifted to a principal $S^{1}$-fibration $BM\,\longrightarrow\, BX$
(see \cite[Proposition 5.4]{HS}), and this induces an exact sequence
\begin{equation}\label{hes}
\xymatrix{
0\ar[r]&\pi_{2}(M,\, x) \ar[r]&\pi_{2}^{orb}(X,\,S^{1}x)\ar[r]&\pi_{1}(S^1,\, 1)\ar[r]^{\psi(x)_{\ast}}&\pi_{1}(M,\, x) \ar[r]&
\pi_{1}^{orb}(X,\, S^{1}x)\ar[r]&1.
}
\end{equation}

\begin{lemma}
Let $(E,\,D)$ be a flat bundle over $M$, and let $\rho\,:\,\pi_{1}(M,\, x)\,\longrightarrow\, {\rm GL}(E_{x})$ be the monodromy
representation for $(E,\,D)$ with $x\,\in\, M_{reg}$.
Then, $E$ is a quasi-regular (respectively, regular) basic vector bundle if and only if the image of the composition of
homomorphisms
$$\rho\circ \psi(x)_{\ast}\,:\,\pi_{1}(S^1,\,1) \,\longrightarrow\, {\rm GL}(E_{x})$$ is finite (respectively, trivial).
\end{lemma}

\begin{proof}
We identify $\Z\,=\,\pi_{1}(S^1,\,1)$ so that for the loop
$\gamma_{x}\,:\, [0,\, r]\, \longrightarrow\, M$ defined by $t\, \longmapsto\, \varphi_{t}(x)$ we have
$[\gamma_{x}]\,=\, \psi(x)_{\ast}(1)\,\in\, \pi_{1}(M, \,x)$.
Then by the definition of $\Phi^{E}$ we have $\rho\circ\psi(x)_{\ast}(1)\,=\, \rho([\gamma])\,=\,\Phi^{E}_{r}$.
If $E$ is a quasi-regular basic bundle, there exists $m\,\in\, \Z$ such that $\Phi^{E}_{mr}\,=\,{\rm Id}$,
and hence $\rho\circ \psi(x)_{\ast}(m)\,=\,{\rm Id}$.

Conversely, if $\rho\circ \psi(x)_{\ast}\,:\,\pi_{1}(S^1,1)\,\longrightarrow\, {\rm GL}(E_{x})$ has a finite image,
then there exists $m\,\in\, \Z$ such that $\rho\circ \psi(x)_{\ast}(m)\,=\,{\rm Id}$, and
hence $\Phi^{E}_{mr}\big\vert_{E_y}
\,=\,{\rm Id}_{E_y}$ for all points $y\, \in\, M_{reg}$.
Since $M_{reg}$ is open dense in $M$, this implies that $\Phi^{E}_{mr}\,=\,{\rm Id}$.

The statement on the regularity follows by setting $m\,=\,1$ in the above arguments.
\end{proof}

\begin{corollary}\label{qqfla}
If the homomorphism $\psi(x)_{\ast}\,:\,\pi_{1}(S^1,\,1)\,\longrightarrow\, \pi_{1}(M,\, x)$ has a finite (respectively, trivial) image,
then every flat bundle over $M$ is quasi-regular (respectively, regular). 
\end{corollary}

\begin{corollary}\label{regflll}
Let $(E,\,D)$ be a flat bundle over $M$, and let $\rho\,:\,\pi_{1}(M,\, x)\,\longrightarrow\,
{\rm GL}(E_{x})$ be the monodromy representation
for $(E,\, D)$ with $x\,\in\, M_{reg}$. Then, $E$ is a regular basic vector bundle if and only if
$\rho$ factors through the quotient group $\pi^{orb}_{1}(X, \,S^{1}x)$ of $\pi_1(M,\, x)$ in \eqref{hes}.
\end{corollary}

\begin{example}
Consider the $2n+1$-dimensional real Heisenberg group 
\[H_{2n+1}(\R)\,=\,\left\{\left(
\begin{array}{ccc}
1& x&z \\
0& I &\,^{t} y\\
0&0&1
\end{array}
\right)\,\, \Big\vert\,\, x,\,y \,\in\, \R^{n},\,\, z\,\in\, \R\right\}
\]
where $I$ is the $n\times n$ unit matrix. This group
$H_{2n+1}(\R)$ admits a left-invariant Sasakian structure.
Thus, for the discrete subgroup $\Gamma\,:=\, H_{2n+1}(\R)\bigcap {\rm GL}_{n+2}(\Z)$, the nilmanifold
$\Gamma\backslash H_{2n+1}(\R)$ is a compact Sasakian manifold.
This $\Gamma\backslash H_{2n+1}(\R)$ is a regular Sasakian manifold. It is straight-forward to
check that the homomorphism
$$\psi(x)_{\ast}\,:\,\pi_{1}(S^1,\,1)\,=\,\Z\,\longrightarrow\, \pi_{1}(M, \,x)\,=\,\Gamma$$
is the natural central extension 
\[\xymatrix{
0\ar[r]&\Z\ar[r]&\pi_{1}(M,\, x) \ar[r]&\Z^{2n}\ar[r]&1.
}
\]
For the canonical representation $$\rho\,:\, \Gamma\,\longrightarrow\, H_{2n+1}(\R)\,\subset\, {\rm GL}_{n+2}(\C)\, ,$$ the image of
$\rho\circ \psi(x)_{\ast}\,:\, \Z\,\longrightarrow\,{\rm GL}_{n+2}(\C)$ is 
\[\left\{\left(
\begin{array}{ccc}
1& 0& z\\
0& I &\,^{t} 0\\
0&0&1
\end{array}
\right)\,\,\Big\vert\,\, z\,\in\, \Z \right\}\]
and hence it is not finite.
Thus the flat vector bundle corresponding to the representation $\rho$ is not quasi-regular.
\end{example}

A holomorphic vector orbibundle over a complex orbifold $(X,\, {\mathcal U})$ consists of holomorphic vector bundles $p_{\alpha}\,:\,
E_{\widetilde{U}_{\alpha}}
\,\longrightarrow\, \widetilde{U}_{\alpha}$, where
$\widetilde{U}_{\alpha}$ is as in the definition of
orbifolds (see \eqref{ual}), together with homomorphisms $\nu_{\alpha}\,:\,\Gamma_{\alpha}
\,\longrightarrow\,{\rm GL}(E_{\widetilde{U}_{\alpha}})$ satisfying the following conditions:
\begin{itemize}
\item $p_{\alpha}(\nu_{\alpha}(\gamma)(b))\,=\,\gamma^{-1}p_{\alpha}(b)$ for
$b\,\in\, E_{\widetilde{U}_{\alpha}}$ and $\gamma\,\in\, \Gamma_{\alpha}$, and

\item for any $\lambda_{\alpha\gamma}\,:\,\widetilde{U}_{\gamma}\,\longrightarrow\, \widetilde{U}_{\alpha}$, there exists
a bundle map $\Lambda_{\alpha\gamma}\,:\,p^{-1}_{\alpha}(\lambda_{\alpha\gamma}(\widetilde{U}_{\gamma}))\,
\longrightarrow\, E_{\gamma}$ such that for $g_{\gamma}\,\in\, \Gamma_{\gamma}$ and $g_{\alpha}\,\in\, \Gamma_{\alpha}$
with $\lambda_{\alpha\gamma}\circ g_{\gamma}\,=\,g_{\alpha}\circ \lambda_{\alpha\gamma}$, the equality
$$\nu_{\gamma}(g_{\gamma})\circ \Lambda_{\alpha\gamma}\,=\,\Lambda_{\alpha\gamma}\circ \nu_{\alpha}(g_{\alpha})$$ holds.
\end{itemize}

Define the holomorphic tangent bundle $T^{1,0}X $ of $(X,\, {\mathcal U})$ by $T^{1,0}\widetilde{U}_{\alpha}$ with 
$\nu_{\alpha}\,:\,\Gamma_{\alpha}\,\longrightarrow\, {\rm GL}(E_{\widetilde{U}_{\alpha}})$ being given by the differential of
the action of $\Gamma_{\alpha}$ on ${U}_{\alpha}$.

For a holomorphic vector orbibundle $E\,\longrightarrow\, X$, a \textit{holomorphic section} of $E$ is a $\Gamma_{\alpha}$-invariant 
holomorphic section of $E_{\widetilde{U}_{\alpha}}$ for every $\alpha$ satisfying a natural compatibility condition with the transition maps 
for $E$. A Hermitian metric on $E$ is a $\Gamma_{\alpha}$-invariant Hermitian metric $h_{\alpha}$ on $E_{\widetilde{U}_{\alpha}}$ for
each $\alpha$ such that a natural compatibility condition with the transition structure for $E$ is satisfied.
A Higgs orbibundle over a complex orbifold $(X,\,{\mathcal U})$ is a pair $(E,\,\theta)$, where $E$ is a
holomorphic vector orbibundle and $\theta \,=\,\{\theta_{\alpha}\}_{\alpha}$ is a holomorphic section
of $(T^{1,0}X)^{\ast}\otimes {\rm End}(E)$ satisfying the condition $\theta_{\alpha} \wedge \theta_{\alpha}\,=\,0$ for every $\alpha$.
For a Higgs orbibundle $(E,\,\theta)$ over $(X,\,{\mathcal U})$, each $(E_{\widetilde{U}_{\alpha}},\,\theta_{\alpha})$ is a Higgs bundle
over the complex manifold $\widetilde{U}_{\alpha}$ and $\theta_{\alpha}\,\in\, A^{1,0}(\widetilde{U}_{\alpha},\,
E_{\widetilde{U}_{\alpha}})$ is $\Gamma_{\alpha}$-invariant.

Assume that the complex orbifold $(X,\,{\mathcal U})$ admits a 
a K\"ahler metric $\{g_{\alpha}\}_{\alpha}$. For a Higgs bundle $(E,\,\theta)$, a hermitian metric
$\{h_{\alpha}\}_{\alpha}$ on $E$ is Hermite-Einstein if $h_{\alpha}$ is a Hermite-Einstein metric on the Higgs bundle
$(E_{\widetilde{U}_{\alpha}},\, \theta)$ over the K\"ahler manifold $(\widetilde{U}_{\alpha}, \,g_{\alpha})$ for every $\alpha$.

Let $(M,\, (T^{1,0}, \,S,\, I),\, (\eta, \,\xi))$ be a compact quasi-regular Sasakian manifold whose period is $r$.
We consider the quotient $X\,=\,M/S^{1}$ with the earlier mentioned canonical complex orbifold structure $(X,\, {\mathcal U})$.
Suppose a holomorphic vector orbibundle $E\,\longrightarrow\, X$ is given.
Then we construct a vector bundle $\widetilde{E}\,\longrightarrow\, M$ as follows.
For each coordinate chart
$(U_{x},\, \widetilde{U}_{x}, \,\Gamma_{x},\, \phi_{x})\,\in \,{\mathcal U}$ as above, consider
$O_{x}\,=\,S^{1}\times_{\Gamma_{x}} \widetilde{U}_{x}$ as an open neighborhood of $x\, \in\, M$ in a natural way. Define
the vector bundle $\widetilde{E}_{O_{x}}\,=\, S^{1}\times_{\Gamma_{x}} E_{\widetilde{U}_{x}}$ on $O_x$.
We can easily check that this is a basic vector bundle, and $\Phi^{\widetilde{E}}$ corresponds to the $S^{1}$-action on
$\widetilde{E}_{O_{x}}\,=\,S^{1}\times_{\Gamma_{x}} E_{\widetilde{U}_{x}}$. Therefore, this $\widetilde E$ is a regular vector bundle over a compact
quasi-regular Sasakian manifold $(M,\, (T^{1,0}, \,S,\, I),\, (\eta, \,\xi))$.
A $\Phi^{\widetilde{E}}$-invariant Hermitian metric $\widetilde{h}$ on $\widetilde{E}$ defines a $\Gamma_{x}$-invariant
Hermitian metric $h_{\widetilde{U}_{x}}$ on $E_{\widetilde{U}_{x}}$, and the family $\{h_{\widetilde{U}_{x}}\}$ is a Hermitian metric on the
orbibundle $E\,\longrightarrow\, X$.

For a Higgs orbibundle $(E,\,\theta)$ over the complex orbifold $(X,\, {\mathcal U})$, we have the regular b-Higgs bundle 
$(\widetilde{E},\,\widetilde{\theta})$ over the compact Sasakian manifold $(M,\, (T^{1,0}, \,S,\, I),\, (\eta, \,\xi))$. If a 
$\Phi^{\widetilde{E}}$-invariant Hermitian metric $\widetilde{h}$ on $\widetilde{E}$ is Hermite-Einstein for $(\widetilde{E},\,
\widetilde{\theta})$, then the 
Hermitian metric $\{h_{\widetilde{U}_{x}}\}_{x\in X}$ on the orbibundle $E\,\longrightarrow\, X$ induced by $\widetilde{h}$ is
Hermite-Einstein for $(E,\,\theta)$. Thus, by Theorem \ref{tthbk0}, we have the following:

\begin{theorem}\label{orbiHE}
For a Higgs orbibundle $(E,\,\theta)$ over the complex orbifold $(X, \,{\mathcal U})$, if 
the corresponding b-Higgs bundle $(\widetilde{E},\,\widetilde{\theta})$ over the compact 
Sasakian manifold $(M,\, (T^{1,0}, \,S,\, I),\, (\eta, \,\xi))$ is stable, then there 
exists a Hermite-Einstein metric for $(E,\,\theta)$.
\end{theorem}

Conversely, given a regular b-holomorphic vector bundle $\widetilde{E}$ over $(M,\, (T^{1,0}, \,S,\, I),\, (\eta, \,\xi))$, we can construct
the following holomorphic vector orbibundle $E\,\longrightarrow\, X$.
For each $O_{x}\,=\,S^{1}\times_{\Gamma_{x}} \widetilde{U}_{x}$ trivialize $\widetilde{E}$ as $\widetilde{E}_{O_{x}}\,=\,O_{x}\times \C^{r}$,
and write $E_{\widetilde{U}_{x}}\,=\, \widetilde{U}_{x} \times \C^{r}$. Let
$\nu_{x}\,:\,\Gamma_{x}\,\longrightarrow\, {\rm GL}(E_{\widetilde{U}_{\alpha}})$ be the homomorphism
defined by $\gamma\, \longmapsto\, \Psi^{\widetilde{E}}_{\gamma}$, where
$\Psi^{\widetilde{E}}\,:\, S^{1}\times\widetilde{E}\,\longrightarrow\, \widetilde{E}$ is the action induced by
$\Phi^{\widetilde{E}}$; here we regard $\Gamma_{x}\,\subset\, S^{1}$.
This construction and the earlier construction together produce an equivalence between the category of regular
b-Higgs bundle $(\widetilde{E},\,\widetilde{\theta})$ over $(M,\, (T^{1,0}, \,S,\, I),\, (\eta, \,\xi))$ and the category of Higgs
orbibundles $(E,\,\theta)$ over the complex orbifold $(X,\, {\mathcal U})$.

By the arguments in Section \ref{secHarm}, we have the following:

\begin{theorem}
For any compact quasi-regular Sasakian manifold $(M,\, (T^{1,0}, \,S,\, I),\, (\eta, \,\xi))$ there is an equivalence of
categories between the following two:
\begin{itemize}
\item the category of semi-simple flat bundles with a regular basic bundle, and

\item the category of Higgs orbibundles $(E,\,\theta)$ over the associated complex orbifold $(X\,=\,M/S^{1},\, {\mathcal U})$ such that the 
associated b-Higgs bundle $(\widetilde{E},\,\widetilde{\theta})$ over the compact Sasakian manifold $(M,\, (T^{1,0}, \,S,\, I),\, (\eta, 
\,\xi))$ is polystable with $$c_{1,B}(\widetilde{E})\,=\,0\qquad {\rm and} \qquad \int_{M} c_{2,B}(\widetilde{E})\wedge(d\eta)^{n-2}
\wedge\eta\,=\,0.$$
\end{itemize}
\end{theorem}

By Corollary \ref{regflll} we have the following:

\begin{corollary}\label{orbirephi}
For any compact quasi-regular Sasakian manifold $(M,\, (T^{1,0}, \,S,\, I),\, (\eta, \,\xi))$ there is an equivalence of
categories between the following two:
\begin{itemize}
\item the category of semi-simple complex representations of $\pi^{orb}_{1}(X,\, S^{1}x)$ for $x\,\in\, M_{reg}$, and

\item the category of Higgs orbibundles $(E,\,\theta)$ over the associated complex orbifold $(X\,=\,M/S^{1},\, {\mathcal U})$ such that 
the associated b-Higgs bundles $(\widetilde{E},\,\widetilde{\theta})$ over the compact Sasakian manifold
$(M,\, (T^{1,0}, \,S,\, I),\, (\eta, \,\xi))$ is polystable with $$c_{1,B}(\widetilde{E})\,=\, 0\qquad {\rm and} \qquad
\int_{M} c_{2,B}(\widetilde{E})\wedge(d\eta)^{n-2}\wedge\eta\,=\,0.$$
\end{itemize}
\end{corollary}

\begin{example}
Let $(M,\, (T^{1,0}, \,S,\, I),\, (\eta, \,\xi))$ be a compact quasi-regular Sasakian manifold. We recall that
the b-holomorphic bundle $T^{1,0}$ corresponds to the holomorphic tangent orbibundle $T^{1,0}X$ over the complex orbifold $X\,=\,M/S^{1}$.
Construct a Higgs orbibundle $(E,\,\theta)$ over the complex orbifold $X$ as follows:
\begin{itemize}
\item $E\,=\,\C_{X}\oplus T^{1,0}X$, where $\C_{X}$ is the trivial holomorphic line orbibundle, and
\item $\theta\,=\,\left(
\begin{array}{cc}
0&0\\
1 & 0
\end{array}
\right)$.
\end{itemize}
This $(E,\,\theta)$ corresponds to the b-Higgs bundle $(\widetilde{E},\,\widetilde{\theta})$ over $(M,\, (T^{1,0}, \,S,\, I),\, (\eta, \,\xi))$
as defined in Example \ref{exTW}. Thus, if $(\widetilde{E},\,\widetilde{\theta})$ is stable, then a Hermite-Einstein metric exists for
the Higgs orbibundle $(E,\,\theta)$ over the complex orbifold $X$.
\end{example}

\begin{example}\label{PSL2g2}
We consider $M\,=\,\Gamma\backslash{\rm PSL}_{2}(\R)$ as in Example \ref{PSL2g}. Then the associated complex orbifold $(X\,=\,M/S^{1},\, {\mathcal 
U})$ is the orbifold Riemann surface $\Gamma\backslash H$, and we have $\pi^{orb}_{1}(X,\, S^{1}x)\,\cong\, \Gamma$. The regular line bundle 
$E^{\tau}_{\sqrt{-1}}\,=\,E^{\tau}\otimes E_{\sqrt{-1}}$ corresponds to a line orbibundle $L$ over the complex orbifold $X$ such that 
$L^{2}\,=\,T^{1,0}X$. Construct a Higgs orbibundle $(E,\,\theta)$ over the complex orbifold $X$ as follows:
\begin{itemize}
\item $E\,=\,L\oplus L^{\ast}$, and
\item $\theta=\left(
\begin{array}{cc}
0&0 \\
1 & 0
\end{array}
\right)$, where we regard $1\,\in\, (T^{1,0}X)^{\ast}\otimes {\rm Hom}(L^{\ast},\, L)\,=\,(T^{1,0}X)^{\ast}\otimes L^{2}\,=\,\C_{X}$.
\end{itemize}
This Higgs orbibundle $(E,\theta)$ corresponds to the b-Higgs bundle 
$(\widetilde{E},\,\widetilde{\theta})$ over the compact Sasakian manifold $(M,\, (T^{1,0}, 
\,S,\, I),\, (\eta, \,\xi))$ as in Example \ref{PSL2g}; in the case where $X$ is smooth, 
this is Hitchin's fundamental example as explained in the introduction. Hence, a 
Hermite-Einstein metric exists for $(E,\, \theta)$, and the monodromy 
representation of the corresponding flat connection coincides with
the canonical homomorphism $\pi^{orb}_{1}(X,\, 
S^{1}x)\,\cong\,\Gamma\,\hookrightarrow\,{\rm SL}_{2}(\C)$.
\end{example}

\begin{remark}
On hyperbolic orbifold Riemann surfaces, the existence of Hermite-Einstein metrics and correspondence between Higgs bundles and 
representations of orbifold fundamental groups are given in \cite{NS,ALS}. These results are derived from the correspondence in 
\cite{Hit,Si1,Si2} on compact smooth Riemann surfaces under the equivariance and rely on the fact that every compact hyperbolic orbifold 
Riemann surface can be realized as the finite group quotient of a compact smooth hyperbolic Riemann surface. Here we do not need this fact. 
\end{remark}

\begin{remark}
Since the complex orbifold $(X\,=\,M/S^{1},\, {\mathcal U})$ can be seen to be a projective 
variety (see Remark \ref{Orbifun}), we have a desingularization $Z\,\longrightarrow\, X$ by 
Hironaka's theorem \cite{Hir}. The correspondence in \cite{BM} is actually closely tied to 
the correspondence between the flat bundles and the Higgs bundles on the smooth projective variety $Z$. The 
topological fundamental group of $X$ is isomorphic to the fundamental group of $Z$ (see 
\cite{Kol}). But the orbifold fundamental group $\pi^{orb}_{1}(X,\, S^{1}x)$ may not be so. 
The correspondence in \cite{BM} does not quite take into account the orbifold fundamental 
group $\pi^{orb}_{1}(X,\, S^{1}x)$.
\end{remark}

\section{Flat bundles over quasi-regular Sasakian manifolds}

\subsection{DG-categories}

\begin{definition}
A category $\mathcal C$ is called a {\em differential graded category} (DG-category for short) if the following conditions hold:
\begin{itemize}
\item $\mathcal C$ is an additive $\C$-linear category;

\item for any objects $U,\, V\,\in\, {\rm Ob}(\mathcal C)$, the space of morphisms ${\rm Hom}(U,\,V)$ admits a cochain complex
structure $({\rm Hom}^{\ast}(U,\,V),\, d)$ such that ${\rm Hom}^{i}(U,\,V) \, =\, 0$ for all $i\,<\,0$;

\item the identity morphism $1\,\in\, {\rm Hom}(U,\,U)$ satisfies the conditions $1\,\in\, {\rm Hom}^{0}(U,\,U)$ and $d(1)\,=\,0$; and

\item for any $U,\,V,\, W\,\in\, {\rm Ob}(\mathcal C)$ and morphisms $f\,\in\, {\rm Hom}^{i}(U,\,V), \, g\,\in\, {\rm Hom}^{j}(V,\,W)$, 
the Leibniz rule
\[d(fg)\,=\,(df )g+(-1)^{i}fdg
\]
holds.
\end{itemize}
For DG-Categories ${\mathcal C}_{1},\, {\mathcal C}_{2}$,
a functor $$F\,:\,{\mathcal C}_{1}\,\longrightarrow\, {\mathcal C}_{2}$$ of DG-categories is a functor of
categories such that $F_{U,V}\,:\, {\rm Hom}(U, \,V)\,\longrightarrow\, {\rm Hom}(FU,\, FV)$ is a morphism of cochain complexes
for all $U,\,V\,\in\, {\rm Ob}({\mathcal C}_{1})$.
\end{definition}

\begin{definition}
Let $\mathcal C$ be a DG-category.
\begin{enumerate}
\item The additive category $E^{0}{\mathcal C}$ is defined as follows:
\begin{itemize}
\item ${\rm Ob}({\mathcal C})\,=\,{\rm Ob}(E^{0}{\mathcal C})$, and

\item ${\rm Hom}(U,\, V)\,=\,H^{0}({\rm Hom}^{0}(U,\, V))$ for all $U,\,V\,\in \,{\rm Ob}(E^{0}{\mathcal C})$.
\end{itemize}
For a functor $F\,:\,{\mathcal C}_{1}\,\longrightarrow\, {\mathcal C}_{2}$ of DG-categories, we denote by $E^{0}(F)$ the functor
$E^{0}({\mathcal C}_{1})\,\longrightarrow\, E^{0}({\mathcal C}_{2})$ which is induced by $F$.

\item An extension in $\mathcal C$ is a diagram 
\[\xymatrix{
M\ar[r]^{a}&U\ar[r]^{b}&N
}
\]
in $\mathcal C$ with $a\,\in\, {\rm Hom}^{0}(M,\, U)$ and $b\,\in\, {\rm Hom}^{0}(U, \,N)$ with $ba\,=\,0$, $da\,=\,0$ and $db\,=\,0$, such that 
there exists a splitting, meaning there is a diagram
\[\xymatrix{
M&\ar[l]^{g}U&\ar[l]^{h}N
}
\]
with $g\,\in\, {\rm Hom}^{0}(U,\, M)$ and $h\in {\rm Hom}^{0}(N, U)$ such that the
conditions $ga\,=\,1$, $bh\,=\,1$, $gh\,=\,0$ and $ag+hb\,=\,1$ hold.

\item
We define the new DG-category $\overline{\mathcal C}$ such that:
\begin{itemize}
\item $ {\rm Ob}(\overline{\mathcal C}) \,=\,\{(U,\, \eta)\,\in\, {\rm Ob}({\mathcal C})\times {\rm Hom}^{1}(U,\, U)\,\big\vert\,
d\eta +\eta^{2}\,=\,0\}$, and

\item ${\rm Hom}(U, \,V)\,=\,{\rm Hom}^{\ast}(U,\, V)$ for all $(U,\,\eta),\, (V,\,\zeta)\,\in\, {\rm Ob}(\overline{\mathcal C})$, 
while the differential $\overline{d}$ satisfies the condition
\[
\overline{d}(f)\,=\,df+ \zeta f-(-1)^{i}f\eta
\]
for all $f\,\in\, {\rm Hom}^{i}(U,\, V)$.
\end{itemize}

\item The DG-category $\widehat{\mathcal C}$ is the full subcategory of $ {\mathcal C}$ whose objects are successive extensions of objects ${\mathcal C}$.
We call $\widehat{\mathcal C}$ the {\em completion } of ${\mathcal C}$.
For a functor $$F\,:\,{\mathcal C}_{1}\,\longrightarrow\, {\mathcal C}_{2}$$ of DG-categories, we denote by
$\widehat{F}$ the functor $\widehat{\mathcal C}_{1}\,\longrightarrow\, \widehat{\mathcal C}_{2}$ of DG-categories which is induced by $F$.
\end{enumerate}
\end{definition}

\begin{proposition}[{\cite{Si2}}]\label{exdg}
Let $F\,:\,{\mathcal C}_{1}\,\longrightarrow\, {\mathcal C}_{2}$ be a functor of DG-categories.
Suppose that $F$ is surjective on isomorphism classes and for any $U,\,V\,\in\, {\mathcal C}_{1}$,
$$F_{U,V}\,:\,{\rm Hom}(U,\, V)\,\longrightarrow\, {\rm Hom}(FU,\, FV)$$ induces an isomorphism of cohomologies.
Then the induced functor $$E^{0}(\widehat{F})\,:\, E^{0}(\widehat{\mathcal C}_{1})\,\longrightarrow\,
E^{0}(\widehat{\mathcal C}_{2})$$ is an equivalence of categories.
\end{proposition}

\begin{example}[{Flat bundles}]\label{FlDG}
Let $(M,\, (T^{1,0}, \,S,\, I),\, (\eta, \,\xi))$ be a compact Sasakian manifold. 
\begin{enumerate}
\item ${\mathcal C}_{dR}$ is the category of all flat bundles on $M$ with
$${\rm Hom}^{\ast}(U,\, V)\,=\,(A^{\ast}(M,\, {\rm Hom}(U,\, V)),\, D).$$

\item ${\mathcal C}_{dR, B}$ is the category of all flat bundles on $M$ with $${\rm Hom}^{\ast}(U,\, V)\,=\,(A^{\ast}_{B}(M,\,{\rm Hom}(U,
\, V)),\, D)\, .$$ We note that $E^{0}({\mathcal C}_{dR})\,=\,E^{0}({\mathcal C}_{dR, B})$.

\item ${\mathcal C}^{s}_{dR}$ (respectively, ${\mathcal C}^{s}_{dR,B}$) is the full sub-category of ${\mathcal C}_{dR}$
(respectively, ${\mathcal C}_{dR,B}$) 
consisting of semi-simple flat bundles. We note that ${\mathcal C}_{dR}$ is naturally equivalent to $\widehat{\mathcal C}^{s}_{dR}$; see 
\cite[Lemma 3.4]{Si2}. But, we cannot say that ${\mathcal C}_{dR,B}$ is naturally equivalent to $\widehat{\mathcal C}^{s}_{dR, B}$ because
an extension in ${\mathcal C}_{dR,B}$ need not split.

\item ${\mathcal C}_{dR, Bh}$ is the full sub-category of ${\mathcal C}_{dR, B}$ consisting of flat bundles admitting a basic Hermitian metric.
Then ${\mathcal C}^{s}_{dR,B}$ is a full sub-category of ${\mathcal C}_{dR, Bh}$; see \cite[Section 4]{BK}.
By taking the orthogonal complement, every extension in ${\mathcal C}_{dR,Bh}$ can be made to split, and hence ${\mathcal C}_{dR,Bh}$
is naturally equivalent to $\widehat{\mathcal C}^{s}_{dR, B}$.

\item Suppose $(M,\, (T^{1,0}, \,S,\, I),\, (\eta, \,\xi))$ is quasi-regular. In this 
case, ${\mathcal C}^{qR}_{dR, B}$ is the full sub-category of ${\mathcal C}_{dR, B}$ 
consisting of flat bundles which are quasi-regular basic bundles; note that from
Lemma \ref{lemSinv} 
it follows that ${\mathcal C}^{qR}_{dR, B}$ is a full sub-category of ${\mathcal C}_{dR, 
Bh}$. Let ${\mathcal C}^{sqR}_{dR, B}$ be the full sub-category of ${\mathcal C}^{qR}_{dR, 
B}$ consisting of semi-simple flat bundles. By the same argument as above, ${\mathcal 
C}^{qR}_{dR, B}$ is naturally isomorphic to $\widehat{\mathcal C}^{sqR}_{dR, B}$.
\end{enumerate}
\end{example}

\begin{example}[{b-Higgs bundles}]\label{HiDG}
Let $(M,\, (T^{1,0}, \,S,\, I),\, (\eta, \,\xi))$ be a compact Sasakian manifold. 
\begin{enumerate}
\item ${\mathcal C}_{Dol,B}$ is the category of b-Higgs bundles $(E, \,\theta)$ admitting filtrations
$$
0\, \subset\, E_1\, \subset\, \cdots \, \subset\, E_{\ell-1}\, \subset\, E_\ell \,=\, E
$$
of sub-b-Higgs bundles such that the b-Higgs bundle $(E_i/E_{i-1},\, \theta)$ is stable and 
$$c_{1,B}(E_i/E_{i-1})\,=\,0\qquad{\rm and} \qquad \int_{M} 
c_{2,B}(E_i/E_{i-1})\wedge(d\eta)^{n-2}\wedge\eta\,=\,0$$ for every $1\, \leq\, i\, \leq\, 
\ell$ (the Higgs field on $E_i/E_{i-1}$ denoted by $\theta$ is again denoted by $\theta$); 
define $ {\rm Hom}^{\ast}(U,\, V)\,=\,(A^{\ast}_{B}(M,\,{\rm Hom}(U,\, V)), 
\,D^{\prime\prime})$.

\item ${\mathcal C}^{s}_{Dol,B}$ is the full sub-category of ${\mathcal C}_{Dol,B}$ consisting of polystable b-Higgs bundles.

\item ${\mathcal C}_{Dol, Bh}$ is the full sub-category of ${\mathcal C}_{Dol, B}$ consisting of objects admitting a basic Hermitian metric. 
This ${\mathcal C}^{s}_{Dol,B}$ is a full sub-category of ${\mathcal C}_{Dol, Bh}$. By taking the orthogonal complement, every extension in 
${\mathcal C}_{Dol,Bh}$ can be made to split, and hence ${\mathcal C}_{Dol,Bh}$ is naturally equivalent to $\widehat{\mathcal C}^{s}_{Dol, B}$.

\item Suppose $(M,\, (T^{1,0}, \,S,\, I),\, (\eta, \,\xi))$ is quasi-regular. In this case, 
${\mathcal C}^{qR}_{Dol, B}$ is the full sub-category of ${\mathcal C}_{Dol, B}$ consisting 
of quasi-regular b-Higgs bundles. Let ${\mathcal C}^{sqR}_{Dol, B}$ be the full 
sub-category of ${\mathcal C}^{qR}_{Dol, B}$ consisting polystable b-Higgs bundles. By the 
same argument as above, ${\mathcal C}^{qR}_{Dol, B}$ is naturally isomorphic to 
$\widehat{\mathcal C}^{sqR}_{Dol, B}$.
\end{enumerate}
\end{example}

\subsection{Extensions of harmonic bundles}

A harmonic bundle over a compact Sasakian manifold
$$(M,\, (T^{1,0}, \,S,\, I),\, (\eta, \,\xi))$$ is a flat bundle $(E,\,D)$ equipped with a 
harmonic metric $h$. Then by Corlette's theorem, $(E,\,D)$ is semi-simple \cite{Cor}. By Section \ref{secHarm}, we have a b-Higgs bundle $(E,\,
\theta)$, and we have $D\,=\,D^{\prime}+D^{\prime\prime}$ on $A^{\ast}_{B}(M,E)$, where $D^{\prime\prime}$ is the canonical operator
associated with the 
b-Higgs bundle $(E,\, \theta)$. We define the cochain complex $A_{D^{\prime}}(M,E)\,=\,\ker D^{\prime}_{\vert A^{\ast}_{B}(M,E)} $ with the 
differential $D^{\prime\prime}$. This is a sub-complex of both $(A^{\ast}_{B}(M,\,E),\,D)$ and $(A^{\ast}_{B}(M,\,E),\,D^{\prime\prime})$.

We define the DG-category ${\mathcal C}^{harm}_{D^{\prime}}$ whose objects are all harmonic
bundles and $${\rm Hom}^{\ast}(U,\, 
V)\,=\,(A^{\ast}_{D^{\prime}}(M,\, {\rm Hom}(U,\, V)),\, D^{\prime\prime}).$$
Then we have the functors $F^{harm}_{dR,B}\,:\, {\mathcal 
C}^{harm}_{D^{\prime},B}\,\longrightarrow\, {\mathcal C}^{s}_{dR,B}$ and $F^{harm}_{Dol,B}\,:\, {\mathcal C}^{harm}_{D^{\prime},B}\,
\longrightarrow\, {\mathcal C}^{s}_{Dol,B}$. 
These functors are surjective on isomorphism classes (see \cite[Section 7]{BK}). By the same statements as in \cite[Lemma 2.2]{Si2} (the so-called 
formality \cite{DGMS}) on basic differential forms with values in harmonic bundles (see \cite[Section 4]{Ka}), we conclude that for any $U,\,V\,\in 
\, {\mathcal C}^{harm}_{D^{\prime},B}$, the homomorphism $F_{U,V}\,:\,{\rm Hom}^{\ast}(U,\, V)\,\longrightarrow\, {\rm Hom}^{\ast}(FU,\,
FV)$ induces an isomorphism of cohomology for $F\,=\,F^{harm}_{dR,B},\, F^{harm}_{Dol,B}$ (see the proof of \cite[Lemma 3.4]{Si2}).
By Proposition \ref{exdg} and the arguments in 
Examples \ref{FlDG}, \ref{HiDG}, we have an equivalence $E^{0}({\mathcal C}_{dR, Bh})\,\cong\, E^{0}({\mathcal C}_{Dol,Bh})$.
That is to say, we have the following result.

\begin{theorem}
For a compact Sasakian manifold $(M,\, (T^{1,0}, \,S,\, I),\, (\eta, \,\xi))$, there exists an equivalence of categories
between the following two:
\begin{itemize}
\item the category of flat bundles on $M$ admitting a basic Hermitian metric, and

\item the category of b-Higgs bundles $(E, \,\theta)$ so that $E$ admitting a basic 
Hermitian metric and also a filtration
$$
0\, \subset\, E_1\, \subset\, \cdots \, \subset\, E_{\ell-1}\, \subset\, E_\ell \,=\, E
$$
of sub-Higgs bundles such that
the Higgs bundle $(E_i/E_{i-1},\,\theta)$ is stable and
$$c_{1,B}(E_i/E_{i-1})\,=\,0\qquad{\rm and} \qquad \int_{M} c_{2,B}(E_i/E_{i-1})\wedge(d\eta)^{n-2}\wedge\eta\,=\,0$$
for every $1\, \leq\, i\, \leq\, \ell$.
\end{itemize}
\end{theorem}

Suppose $(M,\, (T^{1,0}, \,S,\, I),\, (\eta, \,\xi))$ is quasi-regular. Then, by the 
previous section the equivalence $E^{0}({\mathcal C}_{dR, Bh})\,\cong\, E^{0}({\mathcal 
C}_{Dol,Bh})$ can be restricted to an equivalence $ E^{0}({\mathcal 
C}^{sqR}_{dR,B})\,\cong\, E^{0}( {\mathcal C}^{sqR}_{Dol,B})$. Hence we have the following:

\begin{theorem}\label{EQQUA}
For a compact Sasakian manifold $(M,\, (T^{1,0}, \,S,\, I),\, (\eta, \,\xi))$, there exists an equivalence of categories
between the following two:
\begin{itemize}

\item the category of flat bundles on $M$ which are quasi-regular basic bundles, and
 
\item the category of quasi-regular b-Higgs bundles $(E, \,\theta)$ which admit a filtration 
$$
0\, \subset\, E_1\, \subset\, \cdots \, \subset\, E_{\ell-1}\, \subset\, E_\ell \,=\, E
$$
of sub-b-Higgs bundles such that
the b-Higgs bundle $(E_i/E_{i-1},\,\theta)$ is stable and
$$c_{1,B}(E_i/E_{i-1})\,=\,0\qquad{\rm and} \qquad \int_{M} c_{2,B}(E_i/E_{i-1})\wedge(d\eta)^{n-2}\wedge \eta\,=\,0$$
for every $1\, \leq\, i\, \leq\, \ell$.
\end{itemize}
\end{theorem}

By Corollary \ref{qqfla}, we have:

\begin{corollary}
For a compact quasi-regular Sasakian manifold $(M,\, (T^{1,0}, \,S,\, I),\, (\eta, \,\xi))$, if the homomorphism $\psi(x)_{\ast}\,:\,
\pi_{1}(S^1,\,1)\,\longrightarrow\, \pi_{1}(M,\, x)$ has the finite image, then 
there exists an equivalence of categories between the following two:
\begin{itemize}
\item the category of flat bundles on $M$, and

\item the category of quasi-regular b-Higgs bundles $(E, \,\theta)$ which admit a filtration of sub-b-Higgs bundles
$$
0\, \subset\, E_1\, \subset\, \cdots \, \subset\, E_{\ell-1}\, \subset\, E_\ell \,=\, E
$$
such that the b-Higgs bundle $(E_i/E_{i-1},\,\theta)$ is stable and
$$c_{1,B}(E_i/E_{i-1})\,=\,0\qquad{\rm and} \qquad \int_{M} c_{2,B}(E_i/E_{i-1})\wedge(d\eta)^{n-2}\wedge\eta\,=\,0$$
for every $1\, \leq\, i\, \leq\, \ell$.
\end{itemize}
\end{corollary}

By Section \ref{Orbsec}, we have the following result on orbifolds.

\begin{theorem}\label{EQQUAOR}
For a compact quasi-regular Sasakian manifold $(M,\, (T^{1,0}, \,S,\, I),\, (\eta, \,\xi))$, consider the associated orbifold
$(X\,=\,M/S^{1},\, \mathcal{U})$. Then there exists an equivalence of categories between the following two:
\begin{itemize}
\item the category of representations of the orbifold fundamental group $\pi_{1}^{orb}(X)$, and

\item the category of Higgs orbibundles $(E, \,\theta)$ over the complex orbifold $(X\,=\,M/S^{1},\, \mathcal{U})$ such
that the associated b-Higgs bundle $(\widetilde{E}, \,\widetilde{\theta})$ over $(M,\, (T^{1,0}, \,S,\, I),\, (\eta, \,\xi))$
admits a filtration of sub-b-Higgs bundles
$$
0\, \subset\, E_1\, \subset\, \cdots \, \subset\, E_{\ell-1}\, \subset\, E_\ell \,=\, \widetilde{E}
$$
satisfying the conditions that the b-Higgs bundle $(E_i/E_{i-1},\,\theta)$ is stable and 
$$c_{1,B}(E_i/E_{i-1})\,=\,0\qquad{\rm and} \qquad \int_{M} 
c_{2,B}(E_i/E_{i-1})\wedge(d\eta)^{n-2}\wedge\eta\,=\,0$$ for all $1\, \leq\, i\, \leq\, 
\ell$.
\end{itemize}
\end{theorem}

In Theorem \ref{EQQUA}, we can replace the assumption on quasi-regularity by the assumption of regularity.

Restricting the equivalence in Theorem \ref{EQQUA} to the b-holomorphic bundles (i.e., setting $\theta\,=\,0$), by Remark \ref{KHS}
we have the following:

\begin{corollary}
For a quasi-regular compact Sasakian manifold $(M,\, (T^{1,0}M,\, S,\, I),\, (\eta,\, \xi))$, there exists an equivalence of
categories between the following two:
\begin{itemize}
\item the category of flat bundles $(E, \,D)$ which are quasi-regular basic bundles and admit a filtration of sub-flat bundles
$$
0\, \subset\, E_1\, \subset\, \cdots \, \subset\, E_{\ell-1}\, \subset\, E_\ell \,=\, E
$$
satisfying the conditions that the flat bundle $(E_i/E_{i-1},\, D)$ is unitary, and

\item the category of quasi-regular b-holomorphic bundles $E$ admitting filtration of b-holomorphic sub-bundles
$$
0\, \subset\, E_1\, \subset\, \cdots \, \subset\, E_{\ell-1}\, \subset\, E_\ell \,=\, E
$$
such that the holomorphic bundle $E_i/E_{i-1}$ is stable and $$c_{1,B}(E_i/E_{i-1})\,=0\,\qquad{\rm and} \qquad 
\int_{M} c_{2,B}(E_i/E_{i-1})\wedge(d\eta)^{n-2}\wedge\eta\,=\,0$$ for all $1\, \leq\, i\, \leq\, \ell$.
\end{itemize}
\end{corollary}

\section{Numerically flat vector bundles}

Let $E$ be a b-holomorphic vector bundle over a compact Sasakian manifold $$(M,\, (T^{1,0}M,\, S,\, I),\, (\eta,\, \xi)).$$
Given a basic Hermitian metric $h$ on $E$, we say that 
a section $\Theta\,\in\, A_{B}^{1,1}(M, {\rm End}(E))$ is
$$
\Theta \, \geq\, 0
$$
in the sense of Griffiths if writing $\Theta\,=\,\sqrt{-1}\sum a_{jk\lambda\mu}dz_{j}\wedge d\overline{z}_{k} \otimes
e_{\lambda}^{\ast}\otimes e_{\mu}$ for local transversely holomorphic coordinates $(z_{1},\,\cdots,\,z_{n},\, t)$ of $(M,{\mathcal F}_{\xi})$
and a local orthonormal frame $\{e_{\lambda}\}$, we have
$$\sum a_{jk\lambda\mu}X_{j}\overline{X}_{k} v_{\lambda }\overline{v}_{\mu}\,\geq\, 0$$
for every $x\,\in\, M$, $X\,\in\, T_{x}M$ and $v\,\in\, E_{x}$.

\begin{definition}\label{nef}
A holomorphic vector bundle $E$ over a compact Sasakian manifold is called \textit{numerically effective} ({\it nef} for short) if there is
a sequence of basic Hermitian metrics $h_{m}$ on the $m$-the symmetric powers $S^{m}E$ such that for every $\epsilon\, >\, 0$ and $m\,\ge\, 
m_{0}(\epsilon)$,
$$
R^{\nabla^{h_{m}}}(S^{m}E) \, \geq\, - m\epsilon d\eta\otimes {\rm Id}_{S^{m}E}
$$
in the sense of Griffiths.
\end{definition}

We obtain the following results by the same proofs as in \cite[Proposition 1.14, Proposition 1.15]{DPS}.

\begin{itemize}
\item If b-holomorphic vector bundles $E_{1},\, E_{2}$ over a compact Sasakian manifold are nef, then the tensor product $E_{1}\otimes E_{2}$ is 
also nef.

\item For a nef b-holomorphic vector bundle $E$ over a compact Sasakian manifold,
every tensor power $E^{\otimes m}$, $m\, \geq\, 1$, of $E$ is nef.

\item Given a short exact sequence of b-holomorphic vector bundles over a compact Sasakian 
manifold
$$
0\,\longrightarrow\, E_{1}\,\longrightarrow\, E_{2}\,\longrightarrow\, E_{3}
\,\longrightarrow\, 0\, ,$$
we have the following:
\begin{enumerate}
\item if $E_{2}$ is nef, then $E_{3}$ is nef;
\item if $E_{1}$ and $E_{3}$ are nef, then $E_{2}$ is nef;
\item if $E_{2}$ and $\det E^*_{3}$ are nef, then $E_{1}$ is nef.
\end{enumerate}
\end{itemize}

\begin{proposition}\label{nonvan}
Let $E$ be a b-holomorphic vector bundle over a compact Sasakian manifold $(M,\, 
(T^{1,0}M,\, S,\, I),\, (\eta,\, \xi))$. Suppose that $E$ is nef. For any non-zero $\sigma 
\in A^{0}_{B}(M,\, E^*)$ with $\overline\partial_{E^{\ast}} \sigma\,=\,0$,
the section $\sigma$ does not vanish anywhere on $M$.
\end{proposition}

\begin{proof}
The main idea of proof is same as that of the proof of \cite[Proposition 1.16]{DPS}. Consider the 
topological dual $(A^{n-p,n-p}_{B}(M))^*$. We treat $(A^{n-p,n-p}_{B}(M))^*$ as 
$(p,\,p)$-currents like in the complex case. The
positivity of $T\in (A^{n-p,n-p}_{B}(M))^*$ is defined in
the same manner as done in \cite[Chapter III]{De}.

Take the Hermitian metrics $h_{m}$ as in Definition \ref{nef}, and let $h^*_{m}$ be the
corresponding dual metrics on $S^{m}E^*$.
Let
$$T_{m}\,=\,-\frac{1}{2\pi\sqrt{-1}} \partial_{B}\overline{\partial}_{B} \frac{1}{m}\log \Vert \sigma^{m}\Vert_{h_{m}^{*}}.
$$
Then, using the function
$$T_{m}(\alpha)\,=\,\int_{M} T_{m}\wedge \alpha\wedge\eta, \qquad \alpha\,\in\, A^{n-1,n-1}_{B}(M)$$
we regard $T_{m}\,\in\, (A^{n-1,n-1}_{B}(M))^*$.
Now by the same argument as in the proof of \cite[Proposition 1.16]{DPS},
for $m\,\ge\, m_{0}(\epsilon)$, the
current $T_{m}+\epsilon d\eta\in (A^{n-1,n-1}_{B}(M))^*$ is positive, and 
\[\int_{M}(T_{m}+\epsilon d\eta)\wedge (d\eta)^{n-1}\wedge\eta =\epsilon\int_{M} (d\eta)^{n}\wedge\eta.
\]
By the analogy with the standard relation between the mass measure and the trace measure
of a positive current (see \cite[Proposition 1.14, Definition 1.21]{De}), we conclude that
$T_{m}$ converges weakly to zero.

For a Sasakian manifold $(M,\, (T^{1,0}M,\, S,\, I),\, (\eta,\, \xi))$,
Considering K\"ahler potentials of $d\eta$, we can take local coordinates $(z_{1},\,\cdots,
\,z_{n}, t)$ satisfying the following conditions (see \cite{GKN}): 
\begin{itemize}
\item $\xi\,=\,\frac{\partial}{\partial t}$, and

\item there exists a real valued local basic function $K$
such that 
$$\eta\,=\,dt + \sqrt{-1}\sum_{j}\left( \frac{\partial K}{\partial z_j} dz_{j}-
\frac{\partial K}{\partial \overline{z}_j} d\overline{z}_{j}\right).$$
\end{itemize}
With respect to such coordinates, each $\alpha\,\in\, A^{n-1,n-1}_{B}(M)$ is regarded
as a $(n-1,\,n-1)$-form for $(z_{1},\,\cdots,\,z_{n})$ and 
$\int T_{m}\wedge \alpha\wedge\eta\,=\,\int_{z} T_{m}\wedge \alpha\int dt$.
Thus, the restrictions of $T_{m}$ can be seen as currents with complex variables
$(z_{1},\,\cdots,\,z_{n})$. Hence, the
argument on the Lelong number in the proof of \cite[Proposition 1.16]{DPS} is valid,
and the proposition follows.
\end{proof}

\begin{definition}
A b-holomorphic vector bundle $E$ over a compact Sasakian manifold 
is called \textit{numerically flat} if both $E$ and $E^*$ are nef.
\end{definition}

\begin{theorem}\label{nflatst}
Let $E$ be a quasi-regular b-holomorphic vector bundle on a quasi-regular compact Sasakian manifold $(M,\, (T^{1,0}M,\, S,\, I),\, (\eta,\, \xi))$.
If $E$ is numerically flat, then $E$ admits a filtration 
$$
0\, \subset\, E_1\, \subset\, \cdots \, \subset\, E_{\ell-1}\, \subset\, E_\ell \,=\, E
$$
of b-holomorphic sub-bundles such that for each $i$, the b-holomorphic bundle $E_i/E_{i-1}$ 
is stable and $c_{1,B}(E_i/E_{i-1})\,=\,0$.
\end{theorem}

\begin{proof}
Since $E$ is nef, it follows that $\det(E)$ is also nef; recall that
$E^{\otimes r}$ is nef and $\det(E)$ is a direct summand of $E^{\otimes r}$, where
$r\,=\, {\rm rank}(E)$.
Hence, by the definition of nefness, we conclude that
$${\rm deg}(E)\,=\,\int_{M}c_{1,B}(\det(E))\wedge(d\eta)^{n-1}\wedge \eta\,\ge\, 0.$$
Since $E^*$ is nef, we also have $-{\rm deg}(E)\,=\,\int_{M}c_{1,B}(\det(E)^*)
\wedge(d\eta)^{n-1}\wedge \eta\,\ge\, 0$.
These together imply that ${\rm deg}(E)\,=\,0$.

Let ${\mathcal V}\subset {\mathcal O}_{B}(E)$ be a reflexive subsheaf of minimal rank $p\,>\,0$
such that the quotient ${\mathcal O}_{B}(E)/\mathcal V$ is torsion-free
and also ${\rm deg}({\mathcal V})\,=\,0$.
By Lemma \ref{det}, the holomorphic line bundle ${\rm det}(\mathcal V)$ is quasi-regular, in
particular it admits a basic Hermitian metric $\widetilde{h}$.
On the other hand, there is a transversely analytic sub-variety $S\,\subset\, M$ of complex 
co-dimension at least 3 such that ${\mathcal V}$ is given by a b-holomorphic 
sub-bundle $V\,\subset\, E$ on the complement $M\setminus S$.
By the nefness of $E^*$, there exists basic Hermitian metrics $h_{m}$ on
${\rm det}(V)^{\ast}$ such that for every $\epsilon\, >\, 0$ and $m\,\ge\, m_{0}(\epsilon)$,
$$
R^{\nabla^{h_{m}}}({\rm det}(V)^*) \, \geq\, - \epsilon d\eta
$$
on $M\setminus S$. From the above condition that
${\rm deg}({\mathcal V})\,=\,0$, we have
$$-{\rm deg}({\mathcal V})=-\int_{M-S}R^{\nabla^{h_{m}}}({\rm det}(V)^*)\wedge(d\eta)^{n-1}
\wedge \eta\,=\,0.$$
As in the proof of Proposition \ref{nonvan}, regarding $R^{\nabla^{h_{m}}}({\rm det}(V)^*)
\,\in\, (A^{n-1,n-1}_{B}(M))^*$, we conclude that $R^{\nabla^{h_{m}}}$ converges weakly to zero.

It is known that the topological dual $(A^{2n-r}_{B}(M))^{\ast}$ is identified with the space
of basic currents, i.e., currents $T\,\in\, (A^{2n+1-r}(M))^{\ast}$ satisfying $i_{\xi}T
\,=\,0\,=\, {\mathcal L}_{\xi}T$ (see \cite{AE}).
Consider the cochain complex $D^{\ast}(M)\,=\,(A^{2n+1-\ast}(M))^{\ast} $ with the inclusion
map $A^{\ast}(M)\,\hookrightarrow\, D^{\ast}(M)$.
Then $$D^{\ast}_{B}(M)\,=\,(A^{2n-\ast}_{B}(M))^{\ast}\,\subset\, D^{\ast}(M)$$ is a sub-complex.
By the same way as done in \cite[Section 7.2]{BoG}, we have the short exact sequence
\[\xymatrix{
0\ar[r]&D^{\ast}_{B}(M)\ar[r]&D^{\ast}(M)^{\xi}\ar[r]&D^{\ast-1}_{B}(M)\ar[r]&0
}
\]
where $D^{\ast}(M)^{\xi}$ is the sub-complex of $D^{\ast}(M)$ consisting
$T\,\in \,D^{\ast}(M)$ satisfying ${\mathcal L}_{\xi}T\,=\,0$. This short exact sequence
produces a long exact sequence of cohomologies
\[\xymatrix{
\dots \ar[r]&H^{r}(D^{\ast}_{B}(M))\ar[r]&H^{r}(D^{\ast}(M))\ar[r]&H^{r-1}(D^{\ast}_{B}(M))\ar[r]&H^{r+1}(D^{\ast}_{B}(M))\ar[r]&\dots ;
}
\]
note that we have $H^{\ast}(D^{\ast}(M)^{\xi})\,\cong\, H^{\ast}(D^{\ast}(M))$ by averaging
with respect to the $S^{1}$-action.
Since the inclusion map $A^{\ast}(M)\,\hookrightarrow\, D^{\ast}(M)$ induces an
isomorphism of cohomology, and obviously $H^{0}_{B}(M) \,\cong\, H^{0}(D^{\ast}_{B}(M))$,
applying the five lemma to the diagram 
\[\xymatrix{
\dots \ar[r]&H^{r}_{B}(M)\ar[r]\ar[d]&H^{r}(M)\ar[r]\ar[d]&H^{r-1}_{B}(M)\ar[r]\ar[d]&H^{r+1}_{B}(M)\ar[r]\ar[d]&\dots\\
\dots \ar[r]&H^{r}(D^{\ast}_{B}(M))\ar[r]&H^{r}(D^{\ast}(M))\ar[r]&H^{r-1}(D^{\ast}_{B}(M))\ar[r]&H^{r+1}(D^{\ast}_{B}(M))\ar[r]&\dots
}
\]
induced by
\[\xymatrix{0\ar[r]&A^{\ast}_{B}(M)\ar[r]\ar[d]&A^{\ast}(M)^{\xi}\ar[r]\ar[d]&A^{\ast-1}_{B}(M)\ar[r]\ar[d]&0\\
0\ar[r]&D^{\ast}_{B}(M)\ar[r]&D^{\ast}(M)^{\xi}\ar[r]&D^{\ast-1}_{B}(M)\ar[r]&0,
}
\]
we obtain, inductively, that $H^{r}_{B}(M) \,\cong\, H^{r}(D^{\ast}_{B}(M))$ for any
integer $r$.
We have $[R^{\nabla^{h_{m}}}({\rm det}(V))]\,=\,[R^{\nabla^{\widetilde{h}}}({\rm det}
({\mathcal V}))]$ in $ H^{2}(D^{\ast}_{B}(M))$ and this implies
that $[R^{\nabla^{\widetilde{h}}}({\rm det}({\mathcal V}))]\,=\,
c_{1,B}({\rm det}({\mathcal V}))\,=\,0$ in $H^{2}_{B}(M)$.
By the same way as in the K\"ahler case,
changing $\widetilde{h}$ conformally, we can ensure that
$R^{\nabla^{\widetilde{h}}}({\rm det}({\mathcal V}))\,=\,0$,
and hence ${\rm det}({\mathcal V})$ is unitary flat.
 
By Proposition \ref{nonvan}, the canonical map ${\rm det}({\mathcal V})
\,\longrightarrow\, \bigwedge^{{\rm rk}({\mathcal V})}E$ is injective.
Hence, in the same way as done in \cite[Step 2. Proof of Theorem 1. 18]{DPS}, 
we conclude that ${\mathcal V}$ is given by a b-holomorphic 
sub-bundle $V\,\subset\, E$ on $M$ and $V$ is stable.
We also have $c_{1,B}(V)\,=\,c_{1,B}({\rm det}({\mathcal V}))\,=\,0$.
Indeed, \cite[Lemma 1. 20]{DPS} is purely local and the same argument goes
through in our basic holomorphic setting.
Thus, we can reduce $E$ to $E/V $ and the theorem follows inductively.
\end{proof}

\begin{corollary}\label{flatnflat}
 On a quasi-regular compact Sasakian manifold $(M,\, (T^{1,0}M,\, S,\, I),\, (\eta,\, \xi))$,
there is an equivalence of categories between the following two:
\begin{itemize}
\item The category of quasi-regular numerically flat b-holomorphic vector
bundles $E$ satisfying $\int_{M} c_{2,B}(E)\wedge(d\eta)^{n-2}\wedge\eta\,=\,0$.

\item The category of flat bundles $(E, \,D)$ which are quasi-regular basic bundles and
admit a filtration of sub-flat bundles
$$
0\, \subset\, E_1\, \subset\, \cdots \, \subset\, E_{\ell-1}\, \subset\, E_\ell \,=\, E
$$
such that the flat bundle $(E_i/E_{i-1},\,D)$ is unitary for every $1\, \leq\, i\, \leq\, 
\ell$.
\end{itemize}
\end{corollary}



\begin{thebibliography}{ZZZZZZ}
\bibitem{AE}
A. Abouqateb and A. El Kacimi Alaoui, {\it Fonctionnelles invariantes et courants basiques},
Studia Math. {\bf 143} (2000), 199--219.

\bibitem{ALR} A. Adem, J. Leida and Y. Ruan, ``Orbifolds and stringy topology''. Cambridge 
Tracts in Mathematics, 171. Cambridge University Press, Cambridge, 2007.

\bibitem{ALS}D. Alessandrini, G-S. Lee and F. Schaffhauser,
{\it Hitchin components for orbifolds}, Jour. Euro. Math. Soc. {\bf 25} (2023), 1285--1347.

\bibitem{Au} M. Audin, ``Torus actions on symplectic manifolds'',
Progress in Mathematics, 93. Birkh\"auser Verlag, Basel, 2004.

\bibitem{Bai} W. Baily, {\it On the imbedding of V-manifolds in projective space},  Amer. Jour.
Math. {\bf 79} (1957), 403--430.

\bibitem{BH} D. Baraglia and P. Hekmati, {\it A foliated Hitchin-Kobayashi correspondence}, 
Adv. Math. {\bf 408} (2022), Paper No. 108661.

\bibitem{BK} I. Biswas and H. Kasuya, {\it Higgs bundles and
flat connections over compact Sasakian manifolds}, Comm. Math. Phys.
{\bf 385} (2021), 267--290.

\bibitem{BM} I. Biswas and M. Mj, {\it Higgs bundles on Sasakian manifolds}, Int. Math. 
Res. Not. 11 (2018), 3490--3506.

\bibitem{BoG1} C. P. Boyer and K. Galicki, {\it On Sasakian-Einstein geometry}, Internat. J. Math.
{\bf 11} (2000), 873--909.

\bibitem{BoG} C. P. Boyer and K. Galicki, ``Sasakian geometry'', Oxford Mathematical
Monographs. Oxford University Press, Oxford, 2008. 

\bibitem{BGK} C. P. Boyer, K. Galicki and J. Koll\'ar, {\it Einstein metrics on spheres}, 
Ann. of Math. {\bf 162} (2005), 557--580.

\bibitem{BGN} C. P. Boyer, K. Galicki and M. Nakamaye, {\it On the geometry of 
Sasakian-Einstein 5-manifolds}, Math. Ann. {\bf 325} (2003), 485--524.

\bibitem{Cor} K. Corlette, {\it Flat G-bundles with canonical metrics}, Jour. 
Differential Geom. {\bf 28} (1988), 361--382.

\bibitem{DGMS} P. Deligne, P. Griffiths, J. Morgan and D. Sullivan, {\it Real homotopy 
theory of K\"ahler manifolds}, Invent. Math. {\bf 29} (1975), 245--274.

\bibitem{DPS} J.-P. Demailly, T. Peternell and M. Schneider, {\it Compact complex
manifolds with numerically effective tangent bundles}, Jour. Algebraic Geom. {\bf 3}
(1994), 295--346.

\bibitem{De} J.-P. Demailly, ``Complex analytic and differential geometry'', 
http://www-fourier.ujf-grenoble.fr/~demailly/manuscripts/agbook.pdf, 2012.

\bibitem{DonI} S. K. Donaldson, {\it Infinite determinants, stable bundles and curvature},
Duke Math. Jour. {\bf 54} (1987), 231--247.

\bibitem{DK} T. Duchamp and M. Kalka, {\it Deformation theory for holomorphic foliations}, 
J. Differential Geom. {\bf 14} (1979), 317--337.

\bibitem{EKA} A. El Kacimi-Alaoui, {\it Op\'erateurs transversalement elliptiques sur un 
feuilletage riemannien et applications}, Compositio Math. {\bf 73} (1990), no. 1, 57--106.

\bibitem{GKN} M. Godli\'nski, W. Kopczy\'nski and P. Nurowski, {\it Locally Sasakian 
manifolds}, Class. Quantum Grav. {\bf 17} (2000), 105--115.

\bibitem{HS} A. Haefliger and E. Salem, {\it Actions of tori on orbifolds}, Ann. Global 
Anal. Geom. {\bf 9} (1991), 37--59.

\bibitem{Hir} H. Hironaka, {\it Resolution of singularities of an algebraic variety over a 
field of characteristic zero, I, II}, Ann. of Math. {\bf 79} (1964), 109--203; ibid. 
(2) {\bf 79} (1964), 205--326.

\bibitem{Hit} N. J. Hitchin, {\it The self-duality equations on a Riemann surface}, Proc.
London Math. Soc. {\bf 55}(1987), 59--126.

\bibitem{Ka} H. Kasuya, {\it Almost-formality and deformations of representations of the 
fundamental groups of Sasakian manifolds}, Ann. Mat. Pura Appl. {\bf 202} (2023), 1793--1801.

\bibitem{Ko} S. Kobayashi, ``Differential geometry of complex vector bundles'', 
Princeton University Press, Princeton, NJ, Iwanami Shoten, Tokyo, 1987.

\bibitem{Kol} J. Koll\'ar, {\it Shafarevich maps and plurigenera of algebraic varieties},
Invent. Math. {\bf113} (1993), 177--216.

\bibitem{Lu} M. L\"ubke, {\it Stability of Einstein-Hermitian vector bundles},
Manuscripta Math. {\bf 42} (1983), 245--257.

\bibitem{NS} B. Nasatyr and B. Steer, {\it Orbifold Riemann surfaces and the 
Yang-Mills-Higgs equations}, Ann. Scuola Norm. Sup. Pisa {\bf 22} (1995), 595--643.

\bibitem{Pa} S. J. Patterson, {\it On the cohomology of Fuchsian groups}, Glasgow Math. 
J. {\bf 16} (1975), 123--140.

\bibitem{Ra} J. H. Rawnsley, {\it Flat partial connections and holomorphic structures in 
$\mathcal C^{\infty}$ vector bundles}, Proc. Amer. Math. Soc. {\bf 73} (1979), 
391--397.

\bibitem{RC} J. A. Rojo-Carulli, ``Orbifolds and geometric structures'', Ph. D Thesis 
Universidad Complutense de Madrid (2019) available at 
https://eprints.ucm.es/id/eprint/56824/.

\bibitem{Sat}I. Satake, {\it On a generalization of the notion of manifold}, Proc. Nat. 
Acad. Sci. U.S.A. {\bf 42} (1956), 359--363.

\bibitem{Sc} P. Scott, {\it The geometries of $3$-manifolds}, Bull. London Math. Soc.
{\bf 15} (1983), 401--487.

\bibitem{Si1} C. T. Simpson, {\it Constructing variations of Hodge structure using Yang-Mills 
theory and applications to uniformization}, Jour. Amer. Math. Soc. {\bf 1} (1988),
867--918.

\bibitem{Si2} C. T. Simpson, {\it Higgs bundles and local systems}, Inst. Hautes \'Etudes 
Sci. Publ. Math. No. {\bf 75} (1992), 5--95.

\bibitem{Tan} N. Tanaka, ``A Differential Geometric Study on strongly pseudoconvex
CR manifolds'', Lecture Notes in Math., {\bf 9}, Kyoto University, 1975.

\bibitem{UY} K. Uhlenbeck and S.-T. Yau, {\it On the existence of Hermitian-Yang-Mills 
connections in stable vector bundles}, Comm. Pure Appl. Math. {\bf 39} (1986), 
257--293.

\bibitem{Wa} A. W. Wadsley, {\it Geodesic foliations by circles}, J. Differential 
Geometry {\bf 10} (1975), 541--549.

\bibitem{Web} S. Webster, {\it Pseudo-Hermitian structures on a real hypersurface},
J. Differential Geom. {\bf 13} (1978), 25--41.

\bibitem{Za} X. Zhang, {\it Energy properness and Sasakian-Einstein metrics}, Comm. 
Math. Phys. {\bf306} (2011), 229--260.

\end{thebibliography}
\end{document}